\documentclass[1 [leqno,11pt]{amsart}
\usepackage{amssymb, amsmath,amsmath,latexsym,amssymb,amsfonts,amsbsy, amsthm}
\let\pa=\partial
\let\al=\alpha

\let\d=\delta
\let\e=\varepsilon

\let\r=\rho
\let\s=\sigma
\let\f=\frac
\let\p=\psi

\let\D=\Delta

\let\Om=\Omega
\let\wt=\widetilde

\def\cM{{\cal M}}

\def\cM{{\mathcal M}}

\def\na{\nabla}

\def\p{\partial}

\def\dv{\mbox{div}}

\def\dive{\mathop{\rm div}\nolimits}

\def\eqdefa{\buildrel\hbox{\footnotesize def}\over =}

\def\C{\mathop{\bf C\kern 0pt}\nolimits}
\def\DD{\mathop{\bf D\kern 0pt}\nolimits}
\def\K{\mathop{\bf K\kern 0pt}\nolimits}
\def\N{\mathop{\bf N\kern 0pt}\nolimits}
\def\Q{\mathop{\bf Q\kern 0pt}\nolimits}
\def\R{\mathop{\bf R\kern 0pt}\nolimits}

\newcommand{\la}{\lambda}

\newcommand{\beq}{\begin{equation}}
\newcommand{\eeq}{\end{equation}}
\newcommand{\ben}{\begin{eqnarray}}
\newcommand{\een}{\end{eqnarray}}
\newcommand{\beno}{\begin{eqnarray*}}
\newcommand{\eeno}{\end{eqnarray*}}

\newtheorem{thm}{Theorem}[section]
\newtheorem{lem}{Lemma}[section]
\newtheorem{rmk}{Remark}[section]

\renewcommand{\theequation}{\thesection.\arabic{equation}}

\begin{document}
\title[Global unique solvability of inhomogeneous NS equations]
{ Global  unique solvability of  inhomogeneous Navier-Stokes
equations with bounded density }
\author[M. PAICU]{Marius Paicu}
\address [M. PAICU]
{Universit\'e  Bordeaux 1\\
 Institut de Math\'ematiques de Bordeaux\\
F-33405 Talence Cedex, France}
\email{marius.paicu@math.u-bordeaux1.fr}
\author[P. ZHANG]{Ping Zhang}%
\address[P. ZHANG]
 {Academy of
Mathematics $\&$ Systems Science and  Hua Loo-Keng Key Laboratory of
Mathematics, The Chinese Academy of Sciences\\
Beijing 100190, CHINA } \email{zp@amss.ac.cn}
\author[Z. ZHANG]{Zhifei Zhang}\address[Z. ZHANG]
{School of  Mathematical Science, Peking University, Beijing 100871,
P. R. CHINA} \email{zfzhang@math.pku.edu.cn}
\date{Dec. 22, 2012}
\maketitle
\begin{abstract} In this paper,
we prove the global existence and uniqueness of solution to
d-dimensional (for $d=2,3$) incompressible inhomogeneous
Navier-Stokes equations with initial density being bounded from
above and below by some positive constants,  and with initial
velocity $u_0\in H^s(\R^2)$ for $s>0$ in 2-D, or $u_0\in H^1(\R^3)$
satisfying $\|u_0\|_{L^2}\|\na u_0\|_{L^2}$ being sufficiently small
in 3-D. This in particular improves the most recent well-posedness
result in \cite{dm2}, which requires the initial velocity  $u_0\in
H^2(\R^d)$ for the local well-posedness result, and a smallness
condition on the fluctuation of the initial density for the global
well-posedness result.
\end{abstract}



\renewcommand{\theequation}{\thesection.\arabic{equation}}
\setcounter{equation}{0}
\section{Introduction}
In this paper, we consider the  global existence and uniqueness of
the solution to the following d-dimensional (for $d=2,3$)
incompressible inhomogeneous Navier-Stokes equations with initial
density in $L^\infty(\R^d)$ and having a positive lower bound:
\begin{equation}\label{eq:InhomoNS}
\left\{
\begin{array}{ll}
\p_t\rho+u\cdot\na\rho=0,\qquad (t,x)\in\R^+\times\R^d,\\
\rho(\p_tu+u\cdot\na u)-\Delta u+\na p=0, \\
\textrm{div} u=0,\\
(\rho,u)|_{t=0}=(\rho_0,u_0),
\end{array}
\right.
\end{equation}
where $\rho, u$ stand for the density and  velocity of the fluid
respectively, $p$  is a scalar pressure function, and the viscosity
coefficient is supposed to be $1.$ Such a system describes a fluid
which is obtained by mixing two miscible fluids that are
incompressible and that have different densities. It may also
describe a fluid containing a melted substance. One may check
\cite{LP} for the detailed derivation of this system.

Given $0\leq\r_0\in L^\infty(\R^d),$ and $u_0$ satisfying $\dive
u_0=0$, $\sqrt{\r_0}u_0\in L^2(\R^d),$ Lions \cite{LP} (see also
\cite{AKM, Simon} and the references therein for an overview of
results on weak solutions of \eqref{eq:InhomoNS}) proved that
\eqref{eq:InhomoNS} has a global weak solution so that \beno
\f12\|\sqrt{\r(t)}u(t)\|_{L^2}^2+\int_0^t\|\na
u(\tau)\|_{L^2}^2\,d\tau\leq \f12\|\sqrt{\r_0}u_0\|_{L^2}^2. \eeno
Moreover, for any $\al$ and $\beta,$ the Lebesgue measure \beno
\mu\bigl\{x\in\R^d;\ \ \al\leq\r(t,x)\leq\beta\ \bigr\}\quad\mbox{is
independent of }\ t. \eeno In dimension two and under the additional
assumption that $\rho_0$ is bounded below by a positive constant and
$\na u_0\in L^2(\R^2),$ smoother weak solutions may be built. Their
existence stems from a quasi-conservation law involving the norm of
$\na u\in L^\infty((0,T); L^2(\R^2))$ and of $\p_tu, \na p$ and
$\na^2 u\in L^2((0,T); L^2(\R^2))$ for any $T<\infty.$ For both
types of weak solutions however, the problem of uniqueness has not
been solved.

Lady\v zenskaja and Solonnikov  \cite{LS} first addressed the
question of unique solvability of (\ref{eq:InhomoNS}). More
precisely, they  considered the system \eqref{eq:InhomoNS} in a
bounded domain $\Om$ with homogeneous Dirichlet boundary condition
for $u.$ Under the assumption that $u_0\in W^{2-\frac2p,p}(\Om)$
$(p>d)$ is divergence free and vanishes on $\p\Om$ and that $\r_0\in
C^1(\Om)$ is bounded away from zero, then they \cite{LS} proved
\begin{itemize}
\item Global well-posedness in dimension $d=2;$
\item Local well-posedness in dimension $d=3.$ If in addition $u_0$ is small in $W^{2-\frac2p,p}(\Om),$
then global well-posedness holds true.
\end{itemize}
More recently, Danchin \cite{danchin2} established the
well-posedness of the system \eqref{eq:InhomoNS} in the whole space
$\R^d$ for small perturbations of some constant density. Abidi, Gui
and Zhang \cite{AGZ1} investigated the large time decay and global
stability to any global smooth solutions of (\ref{eq:InhomoNS}).

Another important feature of \eqref{eq:InhomoNS} is the scaling
invariant property: if $(\r,u)$ is a solution of \eqref{eq:InhomoNS}
associated to the initial data $(\r_0, u_0)$, then $(\r(\la^2t,\la
x), \la u(\la^2t,\la x))$ is also  a solution of \eqref{eq:InhomoNS}
associated to the initial data $(\r_0(\la x), \la u_0(\la x)).$ A
functional space for the data $(\r_0,u_0)$ or for the solution
$(\r,u)$ is said to be at the scaling of the equation if its norm is
invariant under the above transformation. In this framework, it has
been stated in \cite{abidi, danchin} that for the initial data
$(\r_0, u_0)$ satisfying \beno (\r_0-1)\in
\dot{B}^{\f{d}p}_{p,1}(\R^d),\  u_0\in
\dot{B}^{\f{d}p-1}_{p,1}(\R^d)\ \mbox{with}\ \dive u_0=0 \eeno and
that for a small enough constant $c$ \beno
\|\r_0-1\|_{\dot{B}^{\f{d}p}_{p,1}}+\|u_0\|_{\dot{B}^{\f{d}p-1}_{p,1}}\leq
c, \eeno we have for any $p\in [1, 2d)$
\begin{itemize}
\item existence of global solution $(\r, u, \na p)$ with $\r-1\in
C_b([0,\infty);$ $\dot{B}^{\f{d}p}_{p,1}(\R^d)), $  $u\in
C_b([0,\infty);$ $\dot{B}^{\f{d}p-1}_{p,1}(\R^d)),$ and $\p_tu,
\na^2 u, \na p\in L^1(\R^+;\dot{B}^{\f{d}p-1}_{p,1}(\R^d));$
\item uniqueness in the above space if in addition $p\leq d.$
\end{itemize}
These results have been somewhat extended in \cite{AP} so that $u_0$
belongs to a larger Besov space. Paicu and Zhang \cite{PZ2} further
extended the well-posedness result in \cite{AP} so that even if one
component of the initial velocity is large, \eqref{eq:InhomoNS}
still has a unique global solution. The smallness assumption for the
initial density in \cite{abidi, danchin} has also been removed in
\cite{AGZ2},  and the restriction of $p\in [1,d]$ for uniqueness
result in \cite{abidi, danchin} has been removed recently in
\cite{dm}.

A byproduct in \cite{AGZ2} implies the global existence of solutions
to \eqref{eq:InhomoNS} in 3-D with initial density in
$L^\infty(\R^3)$ and having a positive lower bound, and initial
velocity being sufficiently small in $H^2(\R^3).$ The authors
\cite{dm} proved the global wellposedness of \eqref{eq:InhomoNS}
provided that \beno
\|\r_0-1\|_{\cM(\dot{B}^{\f{d}p-1}_{p,1})}+\|u_0\|_{\dot{B}^{\f{d}p-1}_{p,1}}\leq
c, \eeno for some sufficiently small constant $c,$ where
$\cM(\dot{B}^{\f{d}p-1}_{p,1}(\R^d))$ denotes the multiplier space
of $\dot{B}^{\f{d}p-1}_{p,1}(\R^d).$ This space in particular
includes initial densities having small jumps across a $C^1$
interface.

Again in the scaling invariant framework, the authors \cite{HPZ2}
proved the global existence of weak solutions to \eqref{eq:InhomoNS}
provided that the initial data satisfy the nonlinear smallness
condition:
$$ \bigl(\|\r_0^{-1}-1\|_{L^\infty}+\|u_0^h\|_{\dot{B}^{-1+\frac{d}p}_{p,r}}\bigr)\exp\bigl(
C_r\|u_0^d\|_{\dot{B}^{-1+\frac{d}p}_{p,r}}^{2r}\bigr)\leq c_0\mu
$$ for some positive constants $c_0, C_r$ and $1< p<d,$
$1<r<\infty,$ where $u_0^h=(u_0^1,\cdots,u_0^{d-1})$ and
$u_0=(u_0^h,u_0^d).$ With a little bit more regularity assumption on
the initial velocity, they \cite{HPZ2} also proved the uniqueness of
such solutions.

In general when $\r_0\in L^\infty(\R^d)$ with a positive lower bound
and $u_0\in H^2(\R^d),$ Danchin and Mucha \cite{dm2} proved that the
system \eqref{eq:InhomoNS} has a unique local solution. Furthermore,
with the initial density fluctuation being sufficiently small, for
any initial velocity $u_0\in B^1_{4,2}(\R^2)\cap L^2(\R^2)$ in two
space dimensions, and $u_0\in B^{2-\f2q}_{q,p}(\R^d)$ with
$1<p<\infty, d<q<\infty$ and $2-\f2p\neq\f1q,$ they also proved the
global well-posedness of \eqref{eq:InhomoNS}.

On the other hand, Hoff \cite{Hoff, Hoff2} proved the global
existence of small energy solutions to the isentropic compressible
Navier-Stokes system. The main idea in \cite{Hoff, Hoff2} is that
with appropriate time weight (see Remark \ref{rmk1.2} for details),
one can close the energy estimate for space derivatives of the
velocity field even if the initial velocity only belongs to
$L^2(\R^d).$ Motivated by \cite{Hoff, Hoff2}, we shall investigate
the global well-poseness of \eqref{eq:InhomoNS} with less regular
initial velocity than that in \cite{dm2} and without the small
fluctuation assumption on the initial density. We emphasize that the
Lagrangian idea introduced in \cite{dm, dm2} will also be essential
for the proof of the uniqueness result here.

Our main results in this paper can be listed as follows.

\begin{thm}\label{thm:existenc-2D}
{\sl Let $s>0$. Given the initial data $(\rho_0,u_0)$ satisfying
\beq\label{h.7}
0<c_0\le \rho_0(x)\le C_0<+\infty,\quad u_0\in
H^s(\R^2),
\eeq
the system (\ref{eq:InhomoNS}) has a unique global solution
$(\r, u)$
 such that \beq\label{g.1}
 \begin{split}
&c_0\le \rho(t,x)\le C_0\quad\textrm{ for} \quad (t,x)\in [0,+\infty)\times \R^2,\\
&A_0(t)\le C\|u_0\|_{L^2}^2,\\
&A_1(t)\le C\|u_0\|_{H^s}^2\exp\big\{C\|u_0\|_{L^2}^4\big\},\\
&A_2(t)\le
C\big(1+\|u_0\|_{L^2}^8\big)(\|u_0\|_{H^s}^2+\|u_0\|_{H^s}^4)\exp\big\{C\|u_0\|_{L^2}^4\big\},
\end{split}
\eeq for any $t\in [0,+\infty)$. Here $C$ is a constant depending on $c_0, C_0$, and $A_0(t),
A_1(t)$, and $A_2(t)$ are defined by
\beno\begin{split}
&A_0(T)\eqdefa\f12\sup_{t\in [0,T]}\int_{\R^2}\rho|u(t,x)|^2\,dx+\int_0^T\int_{\R^2}|\na u|^2\,dx dt,\\
&A_1(T)\eqdefa\f12\sup_{t\in [0,T]}\sigma(t)^{1-s}\int_{\R^2}|\na
u(t,x)|^2\,dx+\int_0^T\int_{\R^2}\sigma(t)^{1-s}\bigl(\rho|u_t|^2
+|\na^2u|^2+|\na p|^2\bigr)\,dxdt,\\
&A_2(T)\eqdefa\f12\sup_{t\in
[0,T]}\sigma(t)^{2-s}\int_{\R^2}\bigl(\rho|u_t(t,x)|^2+|\na^2u(t,x)|^2+|\na
p|^2\bigr)\,dx\\
&\qquad\qquad+\int_0^T\int_{\R^2}\sigma(t)^{2-s}|\nabla
u_t|^2\,dxdt,\end{split} \eeno
with $\sigma(t)\eqdefa\min(1,t)$.}
\end{thm}

\begin{thm}\label{thm:existence-3D}
{\sl Given the initial data $(\rho_0,u_0)$ satisfying \beq\label{g.5}
0<c_0\le \rho_0(x)\le C_0<+\infty,\quad u_0\in H^1(\R^3), \eeq there
exists a constant $\varepsilon_0>0$ depending only on  $C_0$ such
that if \beq\label{ass:small} \|u_0\|_{L^2}\|\na u_0\|_{L^2}\le
\varepsilon_0, \eeq
the system (\ref{eq:InhomoNS}) has a unique global solution
$(\rho,u)$ which satisfies \beq\label{g.4}
\begin{split}
&c_0\le \rho(t,x)\le C_0\quad \textrm{for} \quad (t,x)\in [0,+\infty)\times \R^3,\\
&B_0(t)\le \|\rho_0^\f12u_0\|_{L^2}^2,\\
&B_1(t)\le 2\|\na u_0\|_{L^2}^2,\\
&B_2(t)\le C\big(1+\|\na u_0\|_{L^2}^4\big)\|\na
u_0\|_{L^2}^2\exp\bigl\{C(\|u_0\|_{L^2}^2+\|\na
u_0\|_{L^2}^2)\bigr\}, \end{split} \eeq for any $t\in [0,+\infty)$.
Here $C$ is a constant depending on $C_0$, and $B_0(t), B_1(t),$ and
$ B_2(t)$ are defined by \beno
\begin{split}
&B_0(T)\eqdefa\sup_{t\in [0,T]}\int_{\R^3}\rho|u(t,x)|^2\,dx+2\int_0^T\int_{\R^3}|\na u|^2\,dx dt,\\
&B_1(T)\eqdefa\sup_{t\in [0,T]}\int_{\R^3}|\na u(t,x)|^2\,dx+2\int_0^T\int_{\R^3}\bigl(\rho|u_t|^2+|\na^2u|^2+|\na p|^2\bigr)\,dxdt,\\
&B_2(T)\eqdefa\sup_{t\in
[0,T]}\sigma(t)\int_{\R^3}\bigl(\rho|u_t(t,x)|^2+|\na^2u(t,x)|^2+|\na
p|^2\bigr)\,dx+\int_0^T\int_{\R^3}\sigma(t)|\nabla u_t|^2dxdt.
\end{split} \eeno}
\end{thm}

\begin{rmk}
We should point out that we do not need  the lower bound assumption
for the initial density in the existence part of Theorem
\ref{thm:existence-3D}. Indeed, the constant $C$ in \eqref{g.4}  is
independent of $c_0$ in \eqref{g.5}. We can also prove the local
existence and uniqueness solution to the system (\ref{eq:InhomoNS})
even if the initial velocity does not satisfy the smallness
condition \eqref{ass:small}. One may check Remark \ref{rmk2.1} for
details.
\end{rmk}

\begin{rmk}\label{rmk1.2}
 The powers to the weight $\s(t)$ in the energy functionals, $A_i(t), B_i(t)$ for $i=1,2,$
 are
motivated by the following observation: let $s$ be a negative real
number and $(p,r)\in [1,\infty]^2,$ a constant $C$ exists such that
\beq\label{k.1} C^{-1}\|f\|_{\dot{B}^{s}_{p,r}}\leq
\bigl\|\|t^{-\f{s}2}e^{t\D}f\|_{L^p}\bigr\|_{L^r(\R^+;\f{dt}{t})}\leq
C\|f\|_{\dot{B}^{s}_{p,r}}, \eeq (see Theorem 2.34 of \cite{BCD} for
instance). In particular, if $u_0\in H^s(\R^2)$ for $s\in (0,1),$$
\na u_0\in H^{s-1}(\R^2)\hookrightarrow B^{s-1}_{2,\infty}(\R^2)$
and $\na^2u_0 \in B^{s-2}_{2,\infty}(\R^2).$ Then according to
\eqref{k.1}, \beno t^{\f{1-s}2}\|e^{t\D}\na u_0\|_{L^2}\leq
C\|u_0\|_{H^s} \quad\mbox{and}\quad t^{\f{2-s}2}\|e^{t\D}\na^2
u_0\|_{L^2}\leq C\|u_0\|_{H^s}.   \eeno This in some sense explains
the weights in \eqref{g.1}. Similarly, we can explain the weights in
\eqref{g.4}.
\end{rmk}

\begin{rmk} We should point out that we can not directly apply
Theorem 1 of \cite{dm2} concerning the uniqueness of solutions to
(\ref{eq:InhomoNS}) to conclude the uniqueness part of Theorem
\ref{thm:existenc-2D} and Theorem \ref{thm:existence-3D}. Yet the
Lagrangian idea in \cite{dm, dm2} can be successfully applied  to
prove the uniqueness part of both Theorems \ref{thm:existenc-2D} and
\ref{thm:existence-3D}. And the uniqueness result of Germain
\cite{Germain} can not be applied here either. The uniqueness result
of \cite{Germain}  requires the density function satisfying
$\na\r\in L^{\infty}([0,T];L^d(\R^d)),$ but here our density
function only belongs to $L^\infty([0,T]\times\R^d).$ Moreover, the
velocity field in Theorems \ref{thm:existenc-2D} and
\ref{thm:existence-3D} does not satisfy the time growth condition in
\cite{Germain}, especially in Theorem \ref{thm:existenc-2D}.
\end{rmk}

\setcounter{equation}{0}
\section{Global solutions to (\ref{eq:InhomoNS}) with large  bounded
density}\label{sect2}

The purpose of this section is to present the proof to the existence
part of both Theorem \ref{thm:existenc-2D} and Theorem
\ref{thm:existence-3D}.

\subsection{Existence of the solution in 2-D}

\begin{proof}[Proof to the existence part of  Theorem \ref{thm:existenc-2D}] Let $j_\epsilon$ be the standard Friedrich's
mollifier. We define \beno \rho_0^\epsilon=j_\epsilon\ast
\rho_0,\quad u_0^\epsilon=j_\epsilon\ast u_0. \eeno And we choose
$\epsilon$ so small  that \beno \f {c_0} 2\le \rho_0^\epsilon(x)\le
2C_0,\quad x\in \R^2. \eeno With the initial data $(\rho_0^\epsilon,
u_0^\epsilon)$, the system (\ref{eq:InhomoNS}) in 2-D has a unique
global smooth solution $(\rho^\epsilon, u^\epsilon).$  In what
follows, we shall only present uniform energy estimates \eqref{g.1}
for the approximate solutions $(\rho^\epsilon, u^\epsilon)$. Then
the existence part of Theorem \ref{thm:existenc-2D} essentially
follows from \eqref{g.1} for  $(\rho^\epsilon, u^\epsilon)$ and
 a standard compactness argument. The uniqueness part of Theorem \ref{thm:existenc-2D}
 will be proved in Section \ref{sect2.3}.

To simplify the notations, we will omit the superscript $\epsilon$
in what follows. First of all, applying the basic $L^2$ energy
estimate to (\ref{eq:InhomoNS}) gives \beq\label{eq:A0-est}
A_0(t)=\f12\|\rho_0^\f12 u_0\|_{L^2}^2\le C\|u_0\|_{L^2}^2\quad
\textrm{for}\quad t\in \R^+. \eeq While it follows from the
transport equation of (\ref{eq:InhomoNS}) and \eqref{h.7} that
\beq\label{eq:density} c_0\le \rho(t,x)\le C_0\quad
\textrm{for}\quad (t,x)\in \R^+\times \R^2. \eeq

To derive the estimate for $A_1(t),$ we get, by taking the $L^2$
inner product of the momentum equation of  (\ref{eq:InhomoNS}) with
$u_t,$ that \beno \int_{\R^2}\rho|u_t|^2\,dx+\f12\f d
{dt}\int_{\R^2}|\na u|^2\,dx=-\int_{\R^2}\rho(u\cdot \na u)\cdot
u_t\,dx, \eeno from which, we infer
\beq\label{eq:A1-est1}\begin{split}
\widetilde{A}_1(t)\eqdefa&\int_0^t\int_{\R^2}\sigma(\tau)\rho|u_t|^2\,dxd\tau+\f12\sigma(t)\int_{\R^2}|\na u(t,x)|^2\,dx\\
\le &-\int_0^t\int_{\R^2}\sigma(\tau)\rho(u\cdot \na u)\cdot
u_t\,dxd\tau+\int_0^t\int_{\R^2}|\na u|^2\,dxd\tau. \end{split} \eeq
In this subsection, we shall frequently use the following version of
Gagliardo-Nirenberg inequality: \beq \label{eq:GN1}
\|a\|_{L^p(\R^2)}\leq C\|a\|_{L^2(\R^2)}^{\f2p}\|\na
a\|_{L^2(\R^2)}^{1-\f2p}\quad\mbox{for}\quad 2\leq p<\infty. \eeq By
virtue of \eqref{eq:density} and \eqref{eq:GN1}, we obtain \beq
\label{eq:A1-est2}
\begin{split}
&\int_0^t\int_{\R^2}\sigma(\tau)\rho(u\cdot \na u)\cdot u_t\,dxd\tau\\
&\le C\int_0^t\sigma(\tau)\|u(\tau)\|_{L^4}\|\na u(\tau)\|_{L^4}\|\rho^\f12 u_t(\tau)\|_{L^2}\,d\tau\\
&\le C\int_0^t\sigma(\tau)\|u(\tau)\|_{L^2}^\f12\|\na
u(\tau)\|_{L^2}\|\Delta u(\tau)\|_{L^2}^{\f12}
\|\rho^\f12 u_t(\tau)\|_{L^2}\,d\tau\\
&\le C_\delta\int_0^t\sigma(\tau)\|u(\tau)\|_{L^2}\|\na
u(\tau)\|_{L^2}^2\|\Delta u(\tau)\|_{L^2}\,
d\tau+\delta\int_0^t\sigma(\tau)\|\rho^\f12 u_t(\tau)\|_{L^2}^2\,d\tau\\
&\le C_\delta A_0(t)\int_0^t\sigma(\tau)\|\na
u(\tau)\|_{L^2}^4\,d\tau +\delta\int_0^t\sigma(\tau)\|\Delta
u(\tau)\|_{L^2}^2d\tau +\delta \widetilde{A}_1(t),\end{split} \eeq
for any $\delta>0,$ where $C_\d$ is a positive constant so that
$C_\d\to\infty$ as $\d\to 0.$ Whereas thanks to \eqref{eq:InhomoNS},
we write \beq\label{h.2} -\Delta u+\na p=-\rho(u_t+u\cdot\na u),
\eeq which along with   the classical estimate on the Stokes system
ensures that \beno
\begin{split}
\|\na^2 u\|_{L^2}+\|\na p\|_{L^2}\le& C\big(\|\rho u_t\|_{L^2}+\|\rho u\cdot\na u\|_{L^2}\big)\\
\le& C\big(\|\rho u_t\|_{L^2}+\|u\|_{L^2}^\f12\|\na u\|_{L^2}\|\na^2
u\|_{L^2}^\f12\big), \end{split} \eeno  so that
\beq\label{eq:A1-est3} \|\na^2 u\|_{L^2}+\|\na p\|_{L^2}\le
C\bigl(\|\rho^\f12 u_t\|_{L^2}+\|u\|_{L^2}\|\na u\|_{L^2}^2\bigr).
\eeq Substituting \eqref{eq:A1-est3} into \eqref{eq:A1-est2} gives
rise to \beq\label{g.2}
\bigl|\int_0^t\int_{\R^2}\sigma(\tau)\rho(u\cdot \na u)\cdot
u_t\,dxd\tau\bigr|\leq C_\delta A_0(t)\int_0^t\sigma(\tau)\|\na
u(\tau)\|_{L^2}^4\,d\tau  +C\delta \widetilde{A}_1(t). \eeq

Summing up (\ref{eq:A0-est}), \eqref{eq:A1-est1} and
(\ref{eq:A1-est3})-(\ref{g.2}) and taking $\delta$ sufficiently
small, we obtain \beno
\widetilde{A}_1(t)+\int_0^t\int_{\R^2}\sigma(\tau)|\na^2
u|^2\,dx\,d\tau\le C\Bigl(\|u_0\|_{L^2}^2\int_0^t\|\na
u(\tau)\|_{L^2}^2\widetilde{A}_1(\tau)d\tau+\|u_0\|_{L^2}^2\Bigr).
\eeno Applying Gronwall's inequality gives \beq\label{g.3}
\widetilde{A}_1(t)+\int_0^t\int_{\R^2}\sigma(\tau)|\na^2
u|^2\,dx\,d\tau\le C\|u_0\|_{L^2}^2\exp\big\{C\|u_0\|_{L^2}^4\big\}.
\eeq To obtain the estimate of $A_1(t)$, we need to use an
interpolation argument.  For this, we consider the linear momentum
equation \beno \rho(\p_tv+u\cdot\na v)-\Delta v+\na p=0,\quad
v(0,x)=v_0(x). \eeno Then it follows from the same line to the proof
of  (\ref{g.3}) and \eqref{eq:A1-est3} that \beno
\begin{split}
&\int_0^t\int_{\R^2}(\rho|v_t|^2+|\na^2v|^2+|\na
p|^2)\,dxd\tau+\int_{\R^2}|\na v(t,x)|^2\,dx
\le C\|v_0\|_{H^1}^2\exp\big\{C\|u_0\|_{L^2}^4\big\},\\
&\int_0^t\int_{\R^2}\sigma(\tau)(\rho|v_t|^2+|\na^2v|^2+|\na
p|^2)\,dxd\tau+\sigma(t)\int_{\R^2}|\na v(t,x)|^2\,dx \le
C\|v_0\|_{L^2}^2\exp\big\{C\|u_0\|_{L^2}^4\big\}.
\end{split}
\eeno
We define the linear operator $Tv_0=\na v$. The above inequalities tell us that
\beno
&&\|Tv_0\|_{L^2}\le C\exp\big\{C\|u_0\|_{L^2}^4\big\}\|v_0\|_{H^1},\\
&&\|Tv_0\|_{L^2}\le C\sigma(t)^{-\frac 12}\exp\big\{C\|u_0\|_{L^2}^4\big\}\|v_0\|_{L^2},
\eeno
from which and Riesz-Thorin interpolation theorem \cite{Graf}, we infer
\beno
\|\na v\|_{L^2}=\|Tv_0\|_{L^2}\le C\sigma(t)^{\f {-1+s} 2}\exp\big\{C\|u_0\|_{L^2}^4\big\}\|v_0\|_{H^s}.
\eeno
Further,  we define a family of operators $T_zv_0=\sigma(t)^{\frac z 2}\rho^{\f12}\partial_t v(t,x)$ for $\textrm{Re} z\in [0,1]$.
Then we have
\beno
&&\|T_{iy}v_0\|_{L^2(0,t;L^2)}\le C\exp\big\{C\|u_0\|_{L^2}^4\big\}\|v_0\|_{H^1},\\
&&\|T_{1+iy}v_0\|_{L^2(0,t;L^2)}\le C\exp\big\{C\|u_0\|_{L^2}^4\big\}\|v_0\|_{L^2},
\eeno
for any  $y\in \R$. Apply Stein interpolation theorem \cite{Graf} to get
\beno
\big\|\sigma(\tau)^{\f{1-s} 2}\|\rho^\f12v_t\|_{L^2}\big\|_{L^2(0,t)}=\|T_{1-s}v_0\|_{L^2(0,t;L^2)}\le C\exp\big\{C\|u_0\|_{L^2}^4\big\}\|v_0\|_{H^s}.
\eeno
The other terms can be treated in a similar way. Therefore we  have
\beq\label{eq:As-est} A_1(t)\le
C\|u_0\|_{H^s}^2\exp\big\{C\|u_0\|_{L^2}^4\big\}.
\eeq

Finally, we manipulate the $H^2$ energy estimate for $u.$ We first
get, by taking the time derivative to the momentum equation of
(\ref{eq:InhomoNS}), that \beno \rho(u_{tt}+u\cdot\na u_t)-\Delta
u_t+\na p_t=-\rho_t(u_t+u\cdot\na u)-\rho u_t\cdot\na u. \eeno
Taking the $L^2$ inner product of the above equation with $u_t,$
 and then using integration by parts,  we write \beno
 \begin{split}
\f 12&\f d {dt}\int_{\R^2}\rho|u_t|^2\,dx+\int_{\R^2}|\na u_t|^2\,dx\\
&=-\int_{\R^2}\rho_t|u_t|^2\,dx-\int_{\R^2}\rho_t(u\cdot\na u)\cdot
u_t\,dx -\int_{\R^2}\rho (u_t\cdot\na u)\cdot u_t\,dx, \end{split}
\eeno from which, we infer \beq\label{eq:A2-est1}
\begin{split}
\wt{A}_2(t)\eqdefa&\f 12\sigma(t)^{2-s}\int_{\R^2}\rho|u_t|^2\,dx+\int_0^t\sigma(\tau)^{2-s}\int_{\R^2}|\na u_t|^2\,dxd\tau\\
=&-\int_0^t\sigma(\tau)^{2-s}\int_{\R^2}\rho_t|u_t|^2\,dxd\tau-\int_0^t\sigma(\tau)^{2-s}\int_{\R^2}\rho_t(u\cdot\na u)\cdot u_t\,dxd\tau\\
&-\int_0^t\sigma(\tau)^{2-s}\int_{\R^2}\rho (u_t\cdot\na u)\cdot u_t\,dxd\tau+(2-s)\int_0^t\sigma(\tau)^{1-s}\int_{\R^2}\rho|u_t|^2\,dxd\tau\\
\eqdefa& A+B+\frak{C}+D.\end{split} \eeq It is obvious to check that
\beq\label{eq:A2-D} |D|\le (2-s)A_1(t), \eeq and \beq\label{eq:A2-C}
\begin{split}
\frak{C} \le&\int_0^t\sigma(\tau)^{2-s}\|\rho\|_{L^\infty}\|\na u\|_{L^2}\|u_t\|_{L^4}^2\,d\tau\\
\le& C\int_0^t\sigma(\tau)^{2-s}\|\na u\|_{L^2}\|u_t\|_{L^2}\|\na u_t\|_{L^2}\,d\tau\\
\le& C\int_0^t\sigma(\tau)^{2-s}\|\na
u\|_{L^2}^2\|\rho^\f12u_t\|_{L^2}^2\,d\tau+\f14\tilde
A_2(t).\end{split} \eeq

 While noticing that
$\rho_t=-u\cdot\na \rho$, we get by using integration by parts
 and \eqref{eq:GN1} that
\beq\label{eq:A2-A}
\begin{split}
A=&-2\int_0^t\sigma(\tau)^{2-s}\int_{\R^2}\rho u_t\cdot(u\cdot\na u_t)\,dxd\tau\\
\le& 2\int_0^t\sigma(\tau)^{2-s}\|u\|_{L^4}\|\rho\|_{L^\infty}\|\na u_t\|_{L^2}\|u_t\|_{L^4}\,d\tau\\
\le& C\int_0^t\sigma(\tau)^{2-s}\|u\|_{L^2}^\f12\|\na u\|_{L^2}^\f12\|u_t\|_{L^2}^\f12\|\na u_t\|_{L^2}^\f32\,d\tau\\
\le& C\int_0^t\sigma(\tau)^{2-s}\|u\|_{L^2}^2\|\na
u\|_{L^2}^2\|u_t\|_{L^2}^2\,d\tau
+\f14\int_0^t\sigma(\tau)^{2-s}\|\na u_t\|_{L^2}^2\,d\tau\\
\le& C\|u_0\|_{L^2}^2\int_0^t\sigma(\tau)^{2-s}\|\na
u\|_{L^2}^2\|u_t\|_{L^2}^2d\tau+\f14\tilde A_2(t).\end{split} \eeq
Along the same line, we write $B$ as \beno
\begin{split}
B=&\int_0^t\sigma(\tau)^{2-s}\int_{\R^2}\rho(u\cdot\na u)\cdot(u\cdot\na u_t)\,dxd\tau\\
&+\int_0^t\sigma(\tau)^{2-s}\int_{\R^2}\rho ((u\cdot\na u)\cdot\na u)\cdot u_t\,dxd\tau\\
&+\int_0^t\sigma(\tau)^{2-s}\int_{\R^2}\rho ((u\otimes u):\na^2)u\cdot u_t\,dxd\tau\\
\eqdefa& B_1+B_2+B_3.\end{split} \eeno By virtue of H\"{o}lder
inequality and \eqref{eq:GN1}, one has \beno
\begin{split}
B_1\le&\int_0^t\sigma(\tau)^{2-s}\|\rho\|_{L^\infty}\|u\|_{L^8}^2\|\na u\|_{L^4}\|\na u_t\|_{L^2}\,d\tau\\
\le& C\int_0^t\sigma(\tau)^{2-s}\|u\|_{L^2}^\f12\|\na u\|_{L^2}^2\|\Delta u\|_{L^2}^\f12\|\na u_t\|_{L^2}\,d\tau\\
\le& C\int_0^t\sigma(\tau)^{2-s}\|u\|_{L^2}\|\na u\|_{L^2}^4\|\Delta
u\|_{L^2}\,d\tau+\f1 {16}\tilde A_2(t),\end{split} \eeno which along
with (\ref{eq:A1-est3}) implies \beno
\begin{split}
B_1\le&  C\Bigl\{\int_0^t\sigma(\tau)^{2-s}\|\na u\|_{L^2}^2\|\rho^\f12u_t\|_{L^2}^2\,d\tau\\
&+\int_0^t\sigma(\tau)^{2-s}\|u\|_{L^2}^2\|\na u\|_{L^2}^6\,d\tau\Bigr\}+\f1 {16}\tilde A_2(t)\\
\le&  C\int_0^t\sigma(\tau)^{2-s}\|\na
u\|_{L^2}^2\|\rho^\f12u_t\|_{L^2}^2\,d\tau
+C\|u_0\|_{L^2}^4A_1(t)^2+\f1 {16}\tilde A_2(t). \end{split} \eeno
The same argument gives rise to \beno
\begin{split}
B_2\le&\int_0^t\sigma(\tau)^{2-s}\|\rho\|_{L^\infty}^\f12\|u\|_{L^4}\|\na u\|_{L^8}^2\|\rho^\f12u_t\|_{L^2}\,d\tau\\
\le& C\int_0^t\sigma(\tau)^{2-s}\|u\|_{L^2}^\f12\|\na u\|_{L^2}\|\Delta u\|_{L^2}^\f32\|\rho^\f12 u_t\|_{L^2}\,d\tau\\
\le& C\int_0^t\sigma(\tau)^{2-s}\|u\|_{L^2}\|\na
u\|_{L^2}^2\|\rho^\f12 u_t\|_{L^2}^2\,d\tau
+\int_0^t\sigma(\tau)^{2-s}\|\Delta u\|_{L^2}^3\,d\tau\\
\le& C\int_0^t\sigma(\tau)^{2-s}\|u\|_{L^2}\|\na
u\|_{L^2}^2\|\rho^\f12 u_t\|_{L^2}^2\,d\tau
+C\int_0^t\sigma(\tau)^{2-s}\big(\|\rho^\f12u_t\|_{L^2}^3+\|u\|_{L^2}^3\|\na u\|_{L^2}^6\big)\,d\tau\\
\le& C\Bigl(\|u_0\|_{L^2}\int_0^t\sigma(\tau)^{2-s}\|\na
u\|_{L^2}^2\|\rho^\f12 u_t\|_{L^2}^2\,d\tau
+A_1(t)^2+\|u_0\|_{L^2}^5A_1(t)^2\Bigr)+\f1{16}\tilde A_2(t),
\end{split} \eeno and \beno\begin{split}
B_3\le&\int_0^t\sigma(\tau)^{2-s}\|\rho\|_{L^\infty}^\f12\|u\|_{L^\infty}^2\|\na^2 u\|_{L^2}\|\rho^\f12u_t\|_{L^2}\,d\tau\\
\le& C\int_0^t\sigma(\tau)^{2-s}\|u\|_{L^2}\|\Delta u\|_{L^2}^2\|\rho^\f12u_t\|_{L^2}\,d\tau\\
\le&
C\int_0^t\sigma(\tau)^{2-s}\|u\|_{L^2}\|\rho^\f12u_t\|_{L^2}^3\,d\tau
+\int_0^t\sigma(\tau)^{2-s}\|u\|_{L^2}^3\|\na u\|_{L^2}^4\|\rho^\f12u_t\|_{L^2}\,d\tau\\
\le&
C\Bigl(\|u_0\|_{L^2}^2A_1(t)^2+\|u_0\|_{L^2}^8A_1(t)^2+\int_0^t\sigma(\tau)^{2-s}\|\na
u\|_{L^2}^2\|\rho^\f12 u_t\|_{L^2}^2\,d\tau\Bigr) +\f1{16}\tilde
A_2(t).\end{split} \eeno Summing up the above estimates, we conclude
that \beq\label{eq:A2-B} \begin{split}
B\le& C\big(1+\|u_0\|_{L^2}\big)\int_0^t\sigma(\tau)^{2-s}\|\na u\|_{L^2}^2\|\rho^\f12 u_t\|_{L^2}^2\,d\tau\\
&+C\big(1+\|u_0\|_{L^2}^8\big)A_1(t)^2+\f 3 {16}\tilde A_2(t).
\end{split}\eeq

Combining (\ref{eq:A2-est1}) with (\ref{eq:A2-D})-(\ref{eq:A2-B}), we obtain
\beno
\begin{split}
\wt{A}_2(t)\le& C\Bigl\{\big(1+\|u_0\|_{L^2}^2\big)\int_0^t\sigma(\tau)^{2-s}\|\na u\|_{L^2}^2\|\rho^\f12 u_t\|_{L^2}^2\,d\tau\\
&\quad+\big(1+\|u_0\|_{L^2}^8\big)A_1^2(t)+A_1(t)\Bigr\}.
\end{split}
\eeno
Applying Gronwall's inequality and (\ref{eq:As-est}) leads to
\beno
\widetilde{A}_2(t)\le
C\big(1+\|u_0\|_{L^2}^8\big)(\|u_0\|_{H^s}^2+\|u_0\|_{H^s}^4)\exp\big\{C\|u_0\|_{L^2}^4\big\},
\eeno which together with \eqref{eq:A1-est3} ensures that
\beno
\begin{split}
A_2(t)\leq &C\bigl(\wt{A}_2(t)+\|u_0\|_{L^2}^2A_1^2(t)\bigr)\\
\le&
C\big(1+\|u_0\|_{L^2}^8\big)(\|u_0\|_{H^s}^2+\|u_0\|_{H^s}^4)\exp\big\{C\|u_0\|_{L^2}^4\big\},
\end{split}
\eeno This together with \eqref{eq:A0-est} and \eqref{eq:As-est}
completes the proof of \eqref{g.1}.
\end{proof}

\subsection{Existence of the solution in 3-D}

\begin{proof}[Proof to the existence part of Theorem \ref{thm:existence-3D}]
By mollifying the initial density $\rho_0,$ we deduce from
\cite{AGZ2} that (\ref{eq:InhomoNS}) has a unique global solution
$(\rho^\e, u^\e)$ provided that $\e_0$ is small enough in
\eqref{ass:small}. Then the existence part of Theorem
\ref{thm:existence-3D} follows from the uniform estimate \eqref{g.4}
for $(\rho^\e, u^\e)$ and a standard compactness argument. For
simplicity, we only present the {\it a priori}  estimates
\eqref{g.4} for smooth enough solutions $(\rho, u)$ of
(\ref{eq:InhomoNS}).  The uniqueness of such solution will be proved
in Section 3. As a convention in the rest of this section,  we shall
always denote by $C$ a constant depending on $C_0$ in \eqref{g.5},
which may be different from line to line.

First of all, it is easy to check from  (\ref{eq:InhomoNS}) and
\eqref{g.5}
 that \beq\label{eq:A03-est1}\begin{split}
&c_0\le \rho(t,x)\le C_0\quad \textrm{for} \quad (t,x)\in [0,+\infty)\times \R^3,\\
&\int_{\R^3}\rho|u(t,x)|^2\,dx+2\int_0^T\int_{\R^3}|\na u|^2\,dx
dt=\int_{\R^3}\rho_0|u_0|^2\,dx. \end{split} \eeq

While we get by taking the $L^2$ inner product of the momentum
equation to (\ref{eq:InhomoNS}) and $u_t$ that
\beq\label{eq:A13-est1}
2\int_0^t\int_{\R^3}\rho|u_t|^2\,dxd\tau+\int_{\R^3}|\na
u(t,x)|^2\,dx \le \|\na u_0\|_{L^2}^2-\int_0^t\int_{\R^3}\rho(u\cdot
\na u)\cdot u_t\,dxd\tau. \eeq In what follows, we need to use
Gagliardo-Nirenberg inequality \beq\label{eq:GN2}
\|a\|_{L^p(\R^3)}\leq C\|a\|_{L^2(\R^3)}^{\f3p-\f12}\|\na
a\|_{L^2(\R^3)}^{\f32-\f3p}\quad\mbox{for}\quad 2\leq p\leq 6. \eeq
By virtue of \eqref{eq:GN2}, one has
\beno\label{eq:A13-est2}\begin{split}
\int_0^t\int_{\R^3}\rho(u\cdot \na u)\cdot u_t\,dxd\tau\le& \int_0^t\|\rho\|_{L^\infty}^\f12\|u\|_{L^6}\|\na u\|_{L^3}\|\rho^\f12u_t\|_{L^2}\,d\tau\nonumber\\
\le& C\int_0^t\|\na u\|_{L^2}^\f32\|\Delta u\|_{L^2}^\f12\|\rho^\f12u_t\|_{L^2}\,d\tau\nonumber\\
\le& C\int_0^t\|\na u\|_{L^2}^3\|\Delta
u\|_{L^2}\,d\tau+\f14\int_0^t\|\rho^\f12u_t\|_{L^2}^2\,d\tau.
\end{split} \eeno Whereas it follows from the momentum equation of
(\ref{eq:InhomoNS}) and classical estimates on the Stokes system
that \beq\label{eq:A13-est3}\begin{split}
\|\na^2 u\|_{L^2}+\|\na p\|_{L^2}\le& C\big(\|\rho^\f12u_t\|_{L^2}+\|u\cdot\na u\|_{L^2}\big)\\
\le& C\big(\|\rho^\f12u_t\|_{L^2}+\|\na u\|_{L^2}^3\big)+\f12\|\na
^2u\|_{L^2}, \end{split} \eeq  so that \beno
\int_0^t\int_{\R^3}\rho(u\cdot \na u)\cdot u_t\,dxd\tau\le
C\int_0^t\|\na
u\|_{L^2}^6\,d\tau+\f12\int_0^t\|\rho^\f12u_t\|_{L^2}^2\,d\tau,
\eeno from which, (\ref{eq:A13-est1}) and (\ref{eq:A13-est3}), we
infer  \beq\label{B1t} B_1(t)\le \|\na
u_0\|_{L^2}^2+C\|u_0\|_{L^2}^2{B}_1(t)^2. \eeq Hence, as long as we
choose $\varepsilon_0$ small enough in \eqref{g.4}, we obtain the
estimate for $B_1(t)$ in \eqref{g.4}.

We now turn to the estimate of $B_2(t).$ Indeed along the same line
to the proof of  (\ref{eq:A2-est1}), we  have
\beq\label{eq:A23-est1}
\begin{split}
&\f 12\sigma(t)\int_{\R^3}\rho|u_t|^2\,dx+\int_0^t\sigma(\tau)\int_{\R^3}|\na u_t|^2\,dxd\tau\\
&=-\int_0^t\sigma(\tau)\int_{\R^3}\rho_t|u_t|^2\,dxd\tau-\int_0^t\sigma(\tau)\int_{\R^3}\rho_t(u\cdot\na u)\cdot u_t\,dxd\tau\\
&\quad-\int_0^t\sigma(\tau)\int_{\R^3}\rho (u_t\cdot\na u)\cdot u_t\,dxd\tau+\int_0^t\int_{\R^3}\rho|u_t|^2\,dxd\tau\\
&\eqdefa E+F+G+H. \end{split} \eeq It is  obvious to observe that
\beq\label{g.6}  H\le B_1(t)\le 2\|\na u_0\|_{L^2}^2. \eeq Whereas
using $\rho_t=-u\cdot\na \rho$ and integrating by parts, we get, by
applying \eqref{eq:GN2}, that \beq\label{eq:A23-A}
\begin{split}
E=&-2\int_0^t\sigma(\tau)\int_{\R^3}\rho(u\cdot\na u_t)\cdot u_t\,dxd\tau\\
\le& 2\int_0^t\sigma(\tau)\|u\|_{L^\infty}\|\rho\|_{L^\infty}^\f12\|\na u_t\|_{L^2}\|\rho^\f12u_t\|_{L^2}\,d\tau\\
\le& C\int_0^t\sigma(\tau)\|\na u\|_{L^2}^\f12\|\na^2 u\|_{L^2}^\f12\|\na u_t\|_{L^2}\|\rho^\f12u_t\|_{L^2}\,d\tau\\
\le& C\int_0^t\bigl(\|\na u\|_{L^2}^2+\|\na^2
u\|_{L^2}^2\bigr)\sigma(\tau)\|\rho^\f12u_t\|_{L^2}^2\,d\tau
+\f14\int_0^t\sigma(\tau)\|\na u_t\|_{L^2}^2\,d\tau.
\end{split} \eeq Notice that \beno\label{eq:A23-C}
\begin{split}
G \le&\int_0^t\sigma(\tau)\|\rho\|_{L^\infty}^\f12\|\na u\|_{L^3}\|u_t\|_{L^6}\|\rho^\f12 u_t\|_{L^2}\,d\tau\nonumber\\
\le& C\int_0^t\sigma(\tau)\|\na u\|_{L^2}^\f12\|\na^2
u\|_{L^2}^\f12\|\na u_t\|_{L^2}\|\rho^\f12 u_t\|_{L^2}\,d\tau.
\end{split} \eeno \eqref{eq:A23-A} holds also  for $G.$ To deal with  $F$ in \eqref{eq:A23-est1}, we get, by   using
$\rho_t=-u\cdot\na \rho$ once again and integrating by parts, to
write \beno
\begin{split}
F=&\int_0^t\sigma(\tau)\int_{\R^3}\rho(u\cdot\na u)\cdot(u\cdot\na u_t)\,dxd\tau\\
&+\int_0^t\sigma(\tau)\int_{\R^3}\rho ((u\cdot\na u)\cdot\na u)\cdot u_t\,dxd\tau\\
&+\int_0^t\sigma(\tau)\int_{\R^3}\rho ((u\otimes u):\na^2)u\cdot
u_t\,dxd\tau \eqdefa F_1+F_2+F_3. \end{split} \eeno By virtue of
\eqref{g.4} for $B_1(t)$ and \eqref{eq:GN2} , one has \beno
\begin{split}
F_1\le&\int_0^t\sigma(\tau)\|\rho\|_{L^\infty}\|u\|_{L^6}^2\|\na u\|_{L^6}\|\na u_t\|_{L^2}\,d\tau\\
\le& C\int_0^t\sigma(\tau)\|\na u\|_{L^2}^2\|\na^2 u\|_{L^2}\|\na u_t\|_{L^2}\,d\tau\\
\le& C\int_0^t\sigma(\tau)\|\na u\|_{L^2}^4\|\na^2
u\|_{L^2}^2\,d\tau+\f1 {4}B_2(t)\le C\|\na u_0\|_{L^2}^6+\f1
{4}B_2(t).\end{split} \eeno Along the same line, we have \beno
\begin{split}
F_2\le&\int_0^t\sigma(\tau)\|\rho\|_{L^\infty}^\f12\|u\|_{L^6}\|\na u\|_{L^6}^2\|\rho^\f12u_t\|_{L^2}\,d\tau\\
\le& C\int_0^t\sigma(\tau)\|\na u\|_{L^2}\|\na^2 u\|_{L^2}^2\|\rho^\f12 u_t\|_{L^2}\,d\tau\\
\leq&  C\int_0^t\bigl(\|\na u\|_{L^2}^2+\|\na^2
u\|_{L^2}^2\bigr)\sigma(\tau)\bigl(\|\rho^\f12 u_t\|_{L^2}^2+\|\na^2
u\|_{L^2}^2\bigr)\,d\tau.
\end{split} \eeno The same estimate holds for $F_3,$ as \beno
\begin{split}
F_3\le&\int_0^t\sigma(\tau)\|\rho\|_{L^\infty}^\f12\|u\|_{L^\infty}^2\|\na^2 u\|_{L^2}\|\rho^\f12u_t\|_{L^2}\,d\tau\\
\le& C\int_0^t\sigma(\tau)\|\na u\|_{L^2}\|\na^2
u\|_{L^2}^2\|\rho^\f12u_t\|_{L^2}\,d\tau. \end{split} \eeno
Therefore,  we conclude that \beq\label{eq:A23-B}
\begin{split} F\le& C\int_0^t\bigl(\|\na u\|_{L^2}^2+\|\na^2
u\|_{L^2}^2\bigr)\sigma(\tau)\bigl(\|\rho^\f12 u_t\|_{L^2}^2+\|\na^2
u\|_{L^2}^2\bigr)\,d\tau+C\|\na u_0\|_{L^2}^6+\f 1 {4}B_2(t).
\end{split}\eeq

On the other hand, notice from \eqref{eq:A13-est3} that \beno
\s(t)\big(\|\na^2u(t)\|_{L^2}^2+\|\na p\|_{L^2}^2\big)\leq
C\s(t)\bigl(\|\r^{\f12}u_t(t)\|_{L^2}^2+\|\na
u(t)\|_{L^2}^6\bigr)\leq C\bigl(B_2(t)+\|\na u_0\|_{L^2}^6\bigr),
\eeno which together with (\ref{eq:A23-est1})-(\ref{eq:A23-B})
ensures that \beno B_2(t)\leq C\Bigl\{\int_0^t\bigl(\|\na
u\|_{L^2}^2+\|\na^2 u\|_{L^2}^2\bigr)\sigma(\tau)\bigl(\|\rho^\f12
u_t\|_{L^2}^2+\|\na^2 u\|_{L^2}^2\bigr)\,d\tau+\|\na
u_0\|_{L^2}^6+\|\na u_0\|_{L^2}^2\Bigr\},
\eeno
applying Gronwall's inequality gives rise to the estimate of $B_2(t)$ in \eqref{g.4}.
This completes the proof of \eqref{g.4}.
\end{proof}

\begin{rmk}\label{rmk2.1} Along the same line to the derivation of
\eqref{B1t}, we also get \beno
\begin{split}
\f12\f{d}{dt}\|\na
u(t)\|_{L^2}^2+\int_{\R^3}\r|u_t|^2\,dx=&-\int_{\R^3}\r (u\cdot\na
u)\cdot u_t\,dx\\
\leq & C\|\na u\|_{L^2}^6+\f12\int_{\R^3}\r|u_t|^2\,dx,
\end{split}
\eeno which gives \beq\label{large} \f{d}{dt}\|\na
u(t)\|_{L^2}^2+\int_{\R^3}\r|u_t|^2\,dx\leq C\|\na u\|_{L^2}^6. \eeq
Hence if the initial velocity $u_0$ does not satisfy
\eqref{ass:small}, we deduce from \eqref{large} that there exists a
positive time $\frak{T}$ so that \beno \|\na
u\|_{L^\infty_{\frak{T}}(L^2)}^2+\|u_t\|_{L^2_{\frak{T}}(L^2)}^2\leq
C\|\na u_0\|_{L^2}^2. \eeno With the above estimate, we can obtain
the estimate $B_1(t)$ and $B_2(t)$ for $t\leq \frak{T}$ as we did
before. This implies the local existence of solutions to
\eqref{eq:InhomoNS} in 3-D with large data.
\end{rmk}

\smallskip

\renewcommand{\theequation}{\thesection.\arabic{equation}}
\setcounter{equation}{0}
\section{Uniqueness of the solution}\label{sect2.3}

\subsection{More regularity of the solutions}
Before we present the proof to the uniqueness part of both Theorem
\ref{thm:existenc-2D} and Theorem \ref{thm:existence-3D}, we need
the following regularity results for the solutions of
\eqref{eq:InhomoNS} obtained in Section \ref{sect2}.

\begin{lem}\label{lem3.1}
{\sl Let $(\r, u, \na p)$ be the global solution of
\eqref{eq:InhomoNS} obtained in Theorem \ref{thm:existence-3D}. Then
for any $T\in\R^+$, one has \beq\label{h.1}
\begin{split}
&\int_0^T\s(t)\bigl(\|\D u(t)\|_{L^6}^2+\|\na
p(t)\|_{L^6}^2\bigr)\,dt\leq C,\\
&\int_0^T\|\na u(t)\|_{L^\infty}\,dt\leq C\max\bigl(T^{\f14},
T^{\f12}\bigr),\\
&\int_0^T\s(t)^{\f12}\|\na u(t)\|_{L^\infty}^2\,dt\leq C,
\end{split}
\eeq for some constant $C$ depending only on $c_0, C_0$ in
\eqref{g.5} and $\|u_0\|_{H^1}.$}
\end{lem}

\begin{proof} We first get by taking $\dive$ to \eqref{h.2} that
\beno \na p=\na(-\D)^{-1}\dive\bigl\{\r(\p_tu+u\cdot\na u)\bigr\},
\eeno which along with \eqref{g.4} and \eqref{h.2} ensures that \beq
\label{h.3} \|\D u(t)\|_{L^q}+\|\na p(t)\|_{L^q}\leq
C\|(\p_tu+u\cdot\na u)(t)\|_{L^q}\quad\mbox{for any}\ \ q\in
(1,\infty). \eeq However, it follows from Sobolev imbedding theorem
that \beno \|(u\cdot\na u)(t)\|_{L^6}\leq \|u(t)\|_{L^\infty}\|\D
u(t)\|_{L^2}\leq C\|\na u(t)\|_{L^2}^{\f12}\|\D
u(t)\|_{L^2}^{\f32},\eeno from which, \eqref{g.4} and \eqref{h.3},
we infer that
\beno
\begin{split}
&\int_0^T\s(t)\bigl(\|\D u(t)\|_{L^6}^2+\|\na
p\|_{L^6}^2\bigr)\,dt\\
&\leq C\Bigl\{\int_0^T\s(t)\|\na
u_t\|_{L^2}^2\,dt+\sup_{t\in[0,T]}\bigl(\|\na u(t)\|_{L^2}\s(t)\|\D
u(t)\|_{L^2}\bigr)\int_0^T\|\D u(t)\|_{L^2}^2\,dt\Bigr\}\leq C.
\end{split}
\eeno This proves the first part of \eqref{h.1}. Then we deduce from
it, Gagliardo-Nirenberg inequality and \eqref{g.4} that \beno
\begin{split}
\int_0^T\|\na u(t)\|_{L^\infty}\,dt\leq& C\int_0^T\|\D
u\|_{L^2}^{\f12}\|\D u\|_{L^6}^{\f12}\,dt\\
\leq &\|\D u\|_{L^2_T(L^2)}^{\f12}\Bigl(\int_0^T\s(t)\|\D
u(t)\|_{L^6}^2\,dt\Bigr)^{\f14}\Bigl(\int_0^T\s(t)^{-\f12}\,dt\Bigr)^{\f12}\\
\leq &C \Bigl(\int_0^T\s(t)^{-\f12}\,dt\Bigr)^{\f12}\leq
C\max\bigl(T^{\f14}, T^{\f12}\bigr).
\end{split}
\eeno
Along the same line, we can also prove the estimate for
$\int_0^T\s(t)^{\f12}\|\na u(t)\|_{L^\infty}^2\,dt.$
\end{proof}

The 2-D version of the above lemma is more complicated, which we
present as follows.

\begin{lem}\label{lem3.2}
{\sl  Let $(\r, u, \na p)$ be the global solution of
\eqref{eq:InhomoNS} obtained in Theorem \ref{thm:existenc-2D}. Then
for any $T\in\R^+$ and $\al\in [0,1)$, one has \beq\begin{split}
&\int_0^T\s(t)^{1+\al-s}\bigl(\|(\p_t u,\D
u)(t)\|_{L^{\f2{1-\al}}}^2+\|\na
p(t)\|_{L^{\f2{1-\al}}}^2\bigr)\,dt\leq
C,\\
&\int_0^T\|\na u(t)\|_{L^\infty}dt\leq C\max\bigl(T^{\f{s}{2(1+\al)}},
T^{\f12}\bigr),\\
&\int_0^T\s(t)^{1-\frac s {1+\al}}\|\na u(t)\|_{L^\infty}^2dt\leq C,
\end{split} \label{h.8} \eeq
where the constant  $C$ depends on $\al,$ $c_0, C_0$ in \eqref{h.7} and $\|u_0\|_{H^s}.$}
\end{lem}

\begin{proof} Fist of all, for any $\beta\in (0,1),$ we deduce from Gagliardo-Nirenberg
inequality that \beno \|a\|_{L^{\f2{1-\beta}}}\leq
C\|a\|_{L^2}^{1-\beta}\|\na a\|_{L^2}^\beta\quad \mbox{for any}\ \
a\in H^1(\R^2), \eeno which along with \eqref{g.1} ensures that for
any $\al, \beta, \gamma\in (0,1)$ satisfying $\beta+\gamma=1+\al$
\beq\label{h.10}
\begin{split}
\bigl\|&\s(t)^{\f{1-s}2(1+\al)}\|(\r u\cdot\na
u)(t)\|_{L^{\f2{1-\al}}}\bigr\|_{L^2_T}\\
&\leq C_0\bigl\|\s(t)^{\f{1-s}2(1+\al)}\|u(t)\|_{L^{\f2{1-\beta}}}\|\na
u(t)\|_{L^{\f2{1-\gamma}}}\bigr\|_{L^2_T}\\
  &\leq C_0\sup_{t\in
[0,T]}\bigl(\|u(t)\|_{L^2}^{1-\beta}\s(t)^{\f{1-s}2\beta}\|\na
u(t)\|_{L^2}^{\beta}\bigr)\bigl\|\|\na
u(t)\|_{L^2}^{1-\gamma}\s(t)^{\f{1-s}2\gamma}\|\na^2
u(t)\|_{L^2}^{\gamma}\bigr\|_{L^2_T}\\
&\leq C(c_0,C_0,\|u_0\|_{H^s}).
\end{split}
\eeq Along the same line, we obtain the same estimate for
$\bigl\|\s(t)^{\f{1+\al-s}2}\|\p_tu\|_{L^{\f2{1-\al}}}\bigr\|_{L^2_T}.$
On the other hand, it follows from  \eqref{h.3}  that \beno \|\D
u(t)\|_{L^{\f2{1-\al}}}+\|\na p(t)\|_{L^{\f2{1-\al}}}\leq
\|(u_t+u\cdot\na u)(t)\|_{L^{\f2{1-\al}}},\eeno from which and
\eqref{h.10}, we obtain the first inequality of \eqref{h.8}.

Whereas we get, by using Gagliardo-Nirenberg inequality once again,
that \beno
\begin{split}
\int_0^T\|\na u(t)\|_{L^\infty}\,dt\leq & C\int_0^T\|\na
u(t)\|_{L^2}^{\f{\al}{1+\al}}\|\D
u(t)\|_{L^{\f2{1-\al}}}^{\f1{1+\al}}\,dt\\
\leq &C\|\na
u\|_{L^2_T(L^2)}^{\f{\al}{1+\al}}\bigl\|\s(t)^{\f{1+\al-s}2}\|\D
u(t)\|_{L^\f2{1-\al}}\bigr\|_{L^2_T}^{\f1{1+\al}}\Bigl(\int_0^T\s(t)^{-\f{1+\al-s}{1+\al}}\,dt\Bigr)^{\f12},
\end{split}
\eeno which together with \eqref{g.1} ensures the second inequality of
\eqref{h.8}. Along the same line, we can prove the last inequality in \eqref{h.8}.
\end{proof}

\subsection{Lagrangian formulation}
As in \cite{dm, dm2}, we shall apply Lagrangian approach to prove
the  the uniqueness of the solutions. We remark that even with
\eqref{h.1} and \eqref{h.8}, the solution of \eqref{eq:InhomoNS}
obtained in Theorem \ref{thm:existenc-2D} and Theorem
\ref{thm:existence-3D} does not satisfy the assumptions required by
Theorem 1 of \cite{dm2} concerning the uniqueness of solutions to
\eqref{eq:InhomoNS}. Fortunately, the idea used to prove Theorem 1
of \cite{dm2} can be successfully applied here.

Let $(\r,u,p)$ be the  solution of \eqref{eq:InhomoNS} obtained
in  Theorem \ref{thm:existenc-2D} and Theorem
\ref{thm:existence-3D}. Then thanks to \eqref{h.1} and \eqref{h.8},
we can define the trajectory $X(t,y)$ of $u(t,x)$ by
$$\partial_t X(t,y)=u(t,X(t,y)),\qquad X(0,y)=y,$$
which leads to the following relation between the Eulerian
coordinates $x$ and the Lagrangian coordinates $y$: \beq\label{ode1}
X(t,y)=y+\int_0^tu(\tau, X(\tau,y))d\tau. \eeq
Moreover, we deduce from \eqref{h.1} and \eqref{h.8} that we can take $T$ small enough such that
\beq\label{h.27}
\int_0^T\|\na u(t)\|_{L^\infty}dt\leq \f12.
\eeq Then for $t\leq T,$ $X(t,y)$ is invertible with respect to $y$
variables, and we denote by $Y(t,\cdot)$  its inverse mapping. Let
$v(t,y)\eqdefa u(t,x)=u(t,X(t,y))$. One has
\beq\label{h.31}
\begin{split} &\partial_t
v(t,y)=\partial_t u(t,x)+ u(t,x)\cdot\nabla u(t,x),\\
&\partial_{x_i} u^j(t,x)= \partial_{y_k} v^j(t,y)\partial_{x_i} y^k
\quad\mbox{for}\quad x=X(t,y),\ y=Y(t,x). \end{split} \eeq
Let $A(t,y)\eqdefa (\na X(t,y))^{-1}=\na_x Y(t,x)$. So we have
\beq\label{h.12}
 \nabla_x u(t,x)= A^t(t,x)\nabla_y v(t,y)\quad\mbox{ and}\quad
\dive u(t,x)=\dive(A(t,y) v(t,y)). \eeq By the chain rule, we also
have \ben\label{eq:A-div} \dive_y\big(A\cdot\big)=A^t\na_y.\een Here
and in what follows, we always denote $A^t$ the transpose matrix of
$A.$

As in \cite{dm2}, we denote \beq\label{h.13}
\begin{split}
&\na_u\eqdefa A^t\cdot\na_y,\quad
\dive_u\eqdefa\dive(A\cdot)\quad\mbox{and}\quad
\D_u\eqdefa\dive_u\na_u,\\
& \eta(t,y)\eqdefa \r(t, X(t,y)), \quad v(t,y)\eqdefa u(t,
X(t,y))\quad\mbox{and}\quad \Pi(t,y)\eqdefa p(t, X(t,y)).
\end{split}
\eeq Notice that for any $t>0,$  the  solution of
\eqref{eq:InhomoNS} obtained in  Theorem \ref{thm:existenc-2D} and
Theorem \ref{thm:existence-3D} satisfies the smoothness assumption
of Proposition 2 in \cite{dm2}, so that $(\eta, v, \na\Pi)$ defined
by \eqref{h.13} solves
\begin{equation}\label{h.14}
 \quad\left\{\begin{array}{l}
\displaystyle \p_t\eta=0,\\
\displaystyle \eta\pa_t v -\D_u v+ \na_u \Pi=0, \\
\displaystyle \dv_u\, v = 0, \\
\displaystyle (\eta, v)|_{t=0}=(\rho_0, u_0),
\end{array}\right.
\end{equation}
which is  the Lagrangian formulation of \eqref{eq:InhomoNS}.

Now we transform the regularity information of the solution in the Eulerian coordinates into those in the Lagrangian coordinates.

\begin{lem}\label{lem:v-regularity-3D}
{\sl Let $(\r, u, \na p)$ be the global solution of
\eqref{eq:InhomoNS} obtained in Theorem \ref{thm:existence-3D} and
$(\eta,v, \Pi)$ be given by \eqref{h.13}. Then for any $t\leq T$
determined by \eqref{h.27}, one has \beq\label{eq:regularity-v-3}
\begin{split}
&\int_0^t \tau^{\f12}\bigl(\|(\p_t v,\na^2 v)(\tau)\|_{L^3}^2+\|\na
\Pi(\tau)\|_{L^3}^2\bigr)\,d\tau\leq C, \\
&\|\na A\|_{L^\infty_t(L^3)}+\int_0^t\|\na v(\tau)\|_{L^\infty}\,d\tau\leq
Ct^{\f14},\\
&\int_0^t\tau^{\f12}\|\na v(\tau)\|_{L^\infty}^2\,d\tau\leq C,
\end{split}
\eeq
for some constant $C$ depending only on $c_0, C_0$ in
\eqref{g.5} and $\|u_0\|_{H^1}.$}
\end{lem}

\begin{proof} We first deduce from \eqref{ode1}, \eqref{h.27} and
\eqref{h.1} that \beq\label{h.28} \|\na_yX(t,\cdot)\|_{L^\infty}\leq
\exp\bigl\{\int_0^t\|\na_xu(\tau)\|_{L^\infty}\,d\tau\bigr\}\leq
e^{\f12}, \eeq which together with \eqref{h.1} and \eqref{h.13}
implies that \beno \int_0^t\|\na v(t)\|_{L^\infty}\,dt\leq
Ct^{\f14}\quad\mbox{and}\quad \int_0^t\tau^{\f12}\|\na
v(\tau)\|_{L^\infty}^2\,d\tau\leq C. \eeno
Furthermore, thanks to
$\textrm{det}\bigl(\f{\p X(t,y)}{\p y}\bigr)=1,$ and \eqref{g.4},
\eqref{h.1}, one has \beno \|\tau^{\f14}\na\Pi\|_{L^2_t(L^3)}\leq
\|\na p\|_{L^2_t(L^2)}^{\f12}\|\tau^{\f12}\na
p\|_{L^2_t(L^6)}^{\f12}\leq C. \eeno

On the other hand, it follows from the proof of \eqref{h.28} that
\beno
\begin{split}
\|\na_y^2X(t,\cdot)\|_{L^p}\leq &\exp\bigl\{C\int_0^t\|\na_x
u(\tau)\|_{L^\infty}\,d\tau\bigr\}\int_0^t\|\na^2u(\tau,X(\tau,\cdot))\|_{L^p}\,d\tau\\
\leq &C \int_0^t\|\na^2u(\tau,\cdot)\|_{L^p}\,d\tau\quad\mbox{for
any}\ \ p\in [1,\infty].
\end{split}
\eeno In particular, if we take $p=3$ in the above inequality and
use \eqref{h.1} to get \beno \|\na_y^2X(t,\cdot)\|_{L^3}\leq
Ct^{\f14}\|\na^2 u\|_{L^2_t(L^2)}^{\f12}\|\tau^{\f12}\na^2
u\|_{L^2_t(L^6)}^{\f12}\leq Ct^{\f14}, \eeno from which and
\eqref{h.13}, we infer \beno
\begin{split}
\|\tau^{\f14}\na^2v\|_{L^2_t(L^3)}\leq&
C\bigl(\|\tau^{\f14}\na^2_xu(\tau,X(\tau,\cdot))\|_{L^2_t(L^3)}\|\na_y
X\|_{L^\infty_t(L^\infty)}+\|\tau^{\f14}\na_x
u\|_{L^2_t(L^\infty)}\|\na_y^2X\|_{L^\infty_t(L^3)}\bigr)\\
\leq& C\bigl(1+\|\tau^{\f14}\na^2_xu\|_{L^2_t(L^3)}\bigr)\\
\leq& C\bigl(1+\|\na_x^2u\|_{L^2_t(L^2)}^{\f12}\|\tau^{\f12}\na^2_xu\|_{L^2_t(L^6)}^{\f12}\bigr)\leq
C.
\end{split}
\eeno On the other hand, thanks to \eqref{h.27}, we have for $t\leq
T$ \beq\label{eq:A-expansion}
A(t,y)=\na_xY(t,x)=\bigl(Id+(\na_yX(t,y)-Id)\bigr)^{-1}=\sum_{\ell=0}^\infty(-1)^\ell\Bigl(\int_0^t\na_yu(t',X(t',y))\,dt'\Bigr)^\ell,
\eeq for $x=X(t,y),$ so that \beno
\begin{split}
\|\na A\|_{L^\infty_t(L^3)}\leq &
C\|\na_x^2u(\tau,X(\tau,\cdot))\|_{L^1_t(L^3)}\|\na_yX\|_{L^\infty_t(L^\infty)}\leq  Ct^{\f14}.
\end{split}
\eeno

Finally, it follows from \eqref{g.4}, \eqref{h.1} and \eqref{h.31}
that \beno
\begin{split} \|\tau^{\f14}\p_\tau v\|_{L^2_t(L^3)}\leq &  \|\tau^{\f14}\p_\tau u\|_{L^2_t(L^3)}
+\|\tau^{\f14}u\cdot\na u\|_{L^2_t(L^3)}\\
\leq& C\Bigl\{\|\p_\tau
u\|_{L^2_t(L^2)}^{\f12}\|\tau^{\f12}\na\p_\tau
u\|_{L^2_t(L^2)}^{\f12}\\
&\quad+\|\na u\|_{L^\infty_t(L^2)}^{\f12}\|\tau^{\f12}\na^2
u\|_{L^\infty_t(L^2)}^{\f12}\|\na
u\|_{L^2_t(L^2)}^{\f12}\|\na^2u\|_{L^2_t(L^2)}^{\f12}\Bigr\} \leq C.
\end{split}
\eeno
This completes the proof of the lemma.
\end{proof}

\begin{lem}\label{lem:v-regularity-2D}
{\sl Let $(\r, u, \na p)$ be the global solution of
\eqref{eq:InhomoNS} obtained in Theorem \ref{thm:existenc-2D} and
$(\eta, v, \Pi)$ be given by \eqref{h.13}. Then for any $t\leq T$
determined by \eqref{h.27} and $0\le\al<s$, one has
\beq\label{eq:regularity-v-2}
\begin{split}
&\int_0^t\tau^{1+\al-s}\bigl(\|(\p_\tau v, \na^2 v)(\tau)\|_{L^{\f 2{1-\al}}}^2+\|\na
\Pi(\tau)\|_{L^{\f 2{1-\al}}}^2\bigr)\,d\tau\leq C,\\
&\|\na A\|_{L^\infty_t(L^{\f 2{1-\al}})}\leq
Ct^{\f {s-\al} 2},\quad \int_0^t\|\na v(\tau)\|_{L^\infty}\,d\tau\leq
Ct^{\f{s}{2(1+\al)}},\\
&\int_0^t\tau^{1-\frac s {1+\al}}\|\na v(\tau)\|_{L^\infty}^2\,d\tau\leq C,
\end{split}
\eeq for some constant $C$ depending only on $c_0, C_0$ in
\eqref{h.7} and $\|u_0\|_{H^s}.$}
\end{lem}

\begin{proof}

The proof is similar to Lemma \ref{lem:v-regularity-3D}.  We omit the details.
\end{proof}

\subsection{The proof of the uniqueness}

We first recall the following lemma from \cite{dm2}.

\begin{lem}\label{lem3.3}
{\sl Let $\eta\in L^\infty(\R^d)$ be a time independent positive
function, and be bounded away from zero. Let  $R$ satisfy $R_t\in
L^2((0,T)\times\R^d)$ and $\na\dive R\in L^2((0,T)\times\R^d).$ Then
the following system
\begin{equation*}
\left\{
\begin{array}{ll}
\eta\p_tv-\Delta v+\na \Pi=f\qquad (t,x)\in (0,T)\times\R^d,\\
\textrm{div} v=\dive R,\\
v|_{t=0}=v_0,
\end{array}
\right.
\end{equation*}
has a unique solution $(v, \na\Pi)$ such that \beno \|\na
v\|_{L^\infty_T(L^2)}+\|(v_t,\na^2v,\na \Pi)\|_{L^2_T(L^2)}\leq
C\bigl(\|\na v_0\|_{L^2}+\|(f,R_t)\|_{L^2_T(L^2)}+\|\na\dive
R\|_{L^2_T(L^2)}\bigr), \eeno where $C$ depends on $\inf\eta$ and
$\sup\eta,$ but independent of $T.$}
\end{lem}

\begin{proof}[Proof to the uniqueness parts of  Theorems \ref{thm:existenc-2D} and \ref{thm:existence-3D}]
Let  $(\rho_i, u_i, \na p_i), i=1,2,$  be two solutions of
\eqref{eq:InhomoNS}  obtained in Theorem \ref{thm:existenc-2D} and
Theorem \ref{thm:existence-3D}, and $(\eta_i, v_i, \na\Pi), i=1,2,$
be determined by \eqref{h.13}. We denote $A_i\eqdefa A(u_i)$ for
$i=1,2,$ and $\d v\eqdefa v_2-v_1,\d\Pi\eqdefa\Pi_2-\Pi_1$ and $\d
A=A_2-A_1,$ then we deduce from \eqref{h.14} that
\begin{equation}\label{differ}
 \quad\left\{\begin{array}{l}
\displaystyle \r_0\pa_t \d v -\D\d v+ \na \d \Pi= \d f_1 +\d f_2, \\
\displaystyle \dv\, \d v = \dive\d g, \\
\displaystyle \d v|_{t=0}=0,
\end{array}\right.
\end{equation}
with
\beq\label{eq:df-formula}
\begin{split} \d f_1 \eqdefa
&-\bigl[(\na-\na_{u_1})\Pi_1-(\na-\na_{u_2})\Pi_2\bigr]\\
=&-(Id- A_2^t)\na \d \Pi-\d A^t \na
\Pi_1,\\
\d f_2 \eqdefa& -\bigl[(\D-\D_{u_1})v_1-(\D-\D_{u_2})v_2\bigr]\\
 =&
\dive\bigl[(Id-A_2 A_2^t)\na \d v+ (A_1 A_1^t-
A_2 A_2^t)\na v_1],\\
\d g \eqdefa & -(Id-A_1)v_1+(Id-A_2)v_2\\
 =&(Id-A_2)\d v-\d Av_1.
\end{split} \eeq

We denote \beno \d E(t)\eqdefa \|\na\d
v\|_{L^\infty_t(L^2)}+\|(\p_t\d v, \na^2\d v,
\na\d\Pi)\|_{L^2_t(L^2)}. \eeno Then we infer from Lemma
\ref{lem3.3} and \eqref{differ} that \beno \d E(t) \leq C\bigl(\|\d
f_1\|_{L^2_t(L^2)}+\|\d f_2\|_{L^2_t(L^2)}+\|\na \textrm{div}\d
g\|_{L^2_t(L^2)}+\|\p_t\d g\|_{L^2_t(L^2)}\bigr). \eeno We will show
that \ben\label{eq:source-est} \|\d f_1\|_{L^2_t(L^2)}+\|\d
f_2\|_{L^2_t(L^2)}+\|\na\dive\d g\|_{L^2_t(L^2)}+\|\p_t\d
g\|_{L^2_t(L^2)}\le \varepsilon(t)\d E(t), \een where the function
$\varepsilon(t)$ tends to zero as $t$ goes to zero. With
(\ref{eq:source-est}) being granted, we infer that \beno \d E(t)
\leq \varepsilon(t)\d E(t), \eeno which ensures the uniqueness of
solutions obtained in Theorem \ref{thm:existenc-2D} and Theorem
\ref{thm:existence-3D} on a sufficiently small time interval
$[0,T_1].$ The uniqueness on the whole time $[0,\infty)$  can be
obtained by a bootstrap argument.
\end{proof}

Now let us turn to the  proof (\ref{eq:source-est}). Indeed thanks
to \eqref{h.27}, we can take the time $T$ to be small enough so that
\beno \int_0^T \|\na v_i(\tau)\|_{L^\infty}\,d\tau\leq \f12,\quad
i=1,2. \eeno As a convention in the sequel, we shall always assume
that $t\leq T.$ Thanks to \eqref{eq:A-expansion},  we write \beq
\begin{split} & \d A(t)=\bigl(\int_0^t \na \d
v\,d\tau\bigr)\bigl(\sum_{k\geq 1}\sum_{0\leq
j<k}C_1^jC_2^{k-1-j}\bigr)\quad\mbox{with}\quad
C_i(t)\eqdefa\int_0^t\na v_i\,d\tau.\end{split} \label{eq:dA-form}
\eeq

The proof of (\ref{eq:source-est}) will split into the following two
cases. \vspace{0.1cm}

\begin{proof}[ Proof of (\ref{eq:source-est}) in 3-D case]

We first deduce from  \eqref{eq:A-expansion} and
(\ref{eq:regularity-v-3}) that \beno
\|(Id-A_2^t)\na\d\Pi\|_{L^2_t(L^2)}\leq
C\int_0^t\|\na_yv_2(t')\|_{L^\infty}\,dt'\|\na\d\Pi\|_{L^2_t(L^2)}\leq
Ct^{\f 1 4}\d E(t). \eeno While it follows from \eqref{eq:dA-form}
that \beno \|\d A\|_{L^\infty_t(L^6)}\leq C\bigl\|\int_0^\tau|\na\d
v|\,dt'\bigr\|_{L^\infty_t(L^6)}\leq Ct^{\f12}\|\na^2\d
v\|_{L^2_t(L^2)},\eeno which along with \eqref{eq:regularity-v-3}
implies \beno
\begin{split}
\|\d A^t\na \Pi_1\|_{L^2_t(L^2)}\leq &\|\tau^{-\f14}\d
A(\tau)\|_{L^\infty_t(L^6)}\|\tau^{\f14}\na
\Pi_1(\tau)\|_{L^2_t(L^3)}\leq  C t^{\f14}\d E(t),
\end{split}
\eeno
so that thanks to \eqref{eq:df-formula}, we obtain
\beq \label{eq:f1-3d}
\|\d f_1\|_{L^2_t(L^2)}\leq C t^{\f14}\d E(t).
\eeq

Next we handle $\na\dive\d g.$  We first get by applying \eqref{eq:dA-form} that
\ben\label{eq:dA-gradient}
|\na\d A(t)|\le C\int_0^t|\na^2\d v|d\tau+C\int_0^t|\na\d v|d\tau\int_0^t\big(|\na^2 v_1|+|\na^2 v_2|\big)d\tau.
\een
Hence, we have by (\ref{eq:regularity-v-3}) that
\beno
\begin{split}
&\|\na(\d A\na v_1)\|_{L^2_t(L^2)}\\
&\leq C\Bigl\{\bigl\|\int_0^\tau|\na^2\d v|\,d\tau'\na
v_1\bigr\|_{L^2_t(L^2)}+\bigl\|\int_0^\tau|\na\d v|\,d\tau'\na^2
v_1\bigr\|_{L^2_t(L^2)}\\
&\qquad\quad+\bigl\|\int_0^\tau|\na\d
v|\,d\tau'\int_0^\tau\big(|\na^2 v_1|+|\na^2 v_2|\big)d\tau'\na
v_1\bigr\|_{L^2_t(L^2)}\Bigr\}\\
&\leq C\Bigl\{\bigl\|\tau^{-\f14}\int_0^\tau|\na^2\d
v|\,d\tau'\bigr\|_{L^\infty_t(L^2)}\|\tau^{\f14}\na
v_1\|_{L^2_t(L^\infty)}\\
&\qquad\quad+\bigl\|\tau^{-\f14}\int_0^\tau|\na\d
v|\,d\tau'\bigr\|_{L^\infty_t(L^6)}\|\tau^{\f14}(\na^2 v_1,\na^2
v_2)\bigr\|_{L^2_t(L^3)}\big(1+\|\tau^{\f14}\na
v_1\bigr\|_{L^2_t(L^\infty)}\big)\Bigr\}\\
&\leq Ct^{\f14}\d E(t).
\end{split}
\eeno
Along the same line, one has
\beno
\begin{split}
&\|\na\bigl((Id-A_2)\na\d v\bigr)\|_{L^2_t(L^2)}\\
&\leq
C\Bigl\{\bigl\|\int_0^t|\na^2
v_2|\,d\tau'\bigr\|_{L^\infty_t(L^3)}\|\na\d
v\|_{L^2_t(L^6)}+\bigl\|\int_0^t|\na
v_2|\,d\tau'\bigr\|_{L^\infty_t(L^\infty)}\|\na^2\d
v\|_{L^2_t(L^2)}\Bigr\}\\
&\leq Ct^{\f14}\d E(t).
\end{split}
\eeno
Thus, it follows from (\ref{eq:A-div}) and (\ref{eq:df-formula}) that
\beq \label{eq:g-gradient}
\|\na\dive\d
g\|_{L^2_t(L^2)}\leq Ct^{\f14}\d E(t).
\eeq

To deal with $\d f_2,$ we write \ben\label{eq:f2-1}
(A_2A_2^t-A_1A_1^t)\na v_1=\bigl(-\d A(Id-A_2^t)+(Id-A_1)\d
A^t\bigr)\na v_1+(\d A^t+\d A)\na v_1. \een It is easy to check that
\beno
\begin{split}
&\|\dive\bigl[\bigl(\d A(Id-A_2^t)\bigr)\na v_1\bigr]\|_{L^2_t(L^2)}\\
&\leq \|\tau^{-\f14}\na\d
A\|_{L^\infty_t(L^2)}\|Id-A_2\|_{L^\infty_t(L^\infty)}\|\tau^{\f14}\na
v_1\|_{L^2(L^\infty)}\\
&\quad+\|\tau^{-\f14}\d A\|_{L^\infty_t(L^6)}\|\na
A_2\|_{L^\infty_t(L^3)}\|\tau^{\f14}\na
v_1\|_{L^2(L^\infty)}\\
&\quad+\|\tau^{-\f14}\d
A\|_{L^\infty_t(L^6)}\|Id-A_2\|_{L^\infty_t(L^\infty)}\|\tau^{\f14}\na^2
v_1\|_{L^2(L^3)},
\end{split}
\eeno which along with (\ref{eq:dA-gradient}) and \eqref{eq:regularity-v-3} implies that
\beno
\begin{split}
&\|\dive\bigl[\bigl(\d A(Id-A_2^t)\bigr)\na v_1\bigr]\|_{L^2_t(L^2)}\\
&\leq Ct^{\f14}\|\na^2\d v\|_{L^2_t(L^2)}\Bigl\{\|\na
v_2\|_{L^1_t(L^\infty)}\big(\|\tau^{\f14}\na
v_1\|_{L^2(L^\infty)}+\|\tau^{\f14}\na^2v_1\|_{L^2_t(L^3)}\big)\\
&\quad+\|\na A_2\|_{L^\infty_t(L^3)}\|\tau^{\f14}\na
v_1\|_{L^2(L^\infty)} \Bigr\}\leq Ct^{\f12}\d E(t).
\end{split}
\eeno The same estimate holds for $\dive\bigl[(Id-A_1)\d
A^t\bigr)\na v_1\bigr]$ and $\dive\bigl[(\d A^t+\d A)\na v_1\bigr]$.
Hence, \beno \|\dive\bigl[(A_2A_2^t-A_1A_1^t)\na
v_1\bigr]\|_{L^2_t(L^2)}\leq Ct^{\f12}\d E(t). \eeno

To handle $\dive\bigl[(Id-A_2A_2^t)\na\d v\bigr],$ we write
\ben\label{eq:f2-2} (Id-A_2A_2^t)\na\d
v=-\bigl\{(Id-A_2)(Id-A_2^t)-(Id-A_2^t)-(Id-A_2)\bigr\}\na\d v. \een
Then we get, by using (\ref{eq:regularity-v-3}), that \beno
\begin{split}
&\|\dive\bigl[(Id-A_2A_2^t)\na\d v\bigr]\|_{L^2_t(L^2)}\\
&\leq C\Bigl\{\|\na
A_2\|_{L^\infty_t(L^3)}\|Id-A_2\|_{L^\infty_t(L^\infty)}\|\na\d
v\|_{L^2_t(L^6)}+\|(Id-A_2)\|_{L^\infty_t(L^\infty)}\|\na^2\d
v\|_{L^2_t(L^2)}\Bigr\}\\
&\leq C t^{\f14}\d E(t).
\end{split}
\eeno
Summing up, we obtain
\beq\label{eq:f2-3d}
\|\d f_2\|_{L^2_t(L^2)}\leq  C t^{\f14}\d E(t).
\eeq

Finally, let us turn to the estimate of $\|\p_t\d g\|_{L^2_t(L^2)}.$
Thanks to \eqref{eq:dA-form}, we have
\beno
\begin{split}
\|\p_t[\d A v_1]\|_{L^2_t(L^2)}\leq & C\Bigl\{\|(\na\d v)
v_1\|_{L^2_t(L^2)}+\bigl\|\int_0^\tau|\na\d
v|\,d\tau'|(\na
v_1, \na v_2)|v_1\bigr\|_{L^2_t(L^2)}\\
&\quad+\bigl\|\int_0^\tau|\na\d
v|\,d\tau'|\p_tv_1|\bigr\|_{L^2_t(L^2)}\Bigr\}\\
\leq & C\Bigl\{\|\na\d
v\|_{L^\infty_t(L^2)}\|v_1\|_{L^2_t(L^\infty)}+\bigl\|\tau^{-\f14}\int_0^\tau|\na\d
v|\,d\tau'\bigr\|_{L^\infty_t(L^6)}\\
&\quad\times\bigl(\|\tau^{\f14}(\na
v_1, \na v_2)\|_{L^2_t(L^\infty)}\|v_1\bigr\|_{L^\infty_t(L^3)}+\|\tau^{\f14}\p_tv_1\|_{L^2_t(L^3)}\bigr)\Bigr\},
\end{split}
\eeno
which together with \eqref{g.4} and \eqref{eq:regularity-v-3} ensures that
\beno
\begin{split}
\|\p_t[\d A v_1]\|_{L^2_t(L^2)}\leq & C\Bigl\{\|\na
u_1\|_{L^2_t(L^2)}^{\f12}\|\D u_1\|_{L^2_t(L^2)}^{\f12}\|\na\d
v\|_{L^\infty_t(L^2)}+t^{\f14}\|\na^2\d
v\|_{L^2_t(L^2)}\\
&\quad\times\bigl(\|\tau^{\f14}(\na
v_1, \na v_2)\|_{L^2_t(L^\infty)}\|u_1\|_{L^\infty_t(L^2)}^{\f12}\|\na
u_1\|_{L^\infty_t(L^2)}^{\f12}+\|\tau^{\f14}\p_t
v_1\|_{L^2_t(L^3)}\bigr)\Bigr\}\\
\leq &Ct^{\f14}\|\na\d v\|_{L^\infty_t(L^2)}+Ct^{\f14}\|\na^2\d
v\|_{L^2_t(L^2)}\le Ct^{\f14}\d E(t).
\end{split}
\eeno
Along the same line, one has
\beno
\begin{split}
\|\p_t[(Id-A_2)\d v]\|_{L^2_t(L^2)} \leq& C\Bigl\{\|\na
v_2\d v\|_{L^2_t(L^2)}+\bigl\|\int_0^\tau|\na v_2|\,d\tau'|\na
v_2||\d v|\bigr\|_{L^2_t(L^2)}\\
&\qquad+\bigl\|\int_0^\tau|\na v_2|\,d\tau'|\p_t\d
v|\bigr\|_{L^2_t(L^2)}\Bigr\}\\
\leq &C\Bigl\{\|\d v\|_{L^\infty_t(L^6)}\|\na v_2\|_{L^2_t(L^3)}+
\|\na v_2\|_{L^1_t(L^\infty)}\\
&\qquad\times\bigl(\|\na v_2\|_{L^2_t(L^3)}\|\d
v\|_{L^\infty_t(L^6)}+\|\p_t\d v\|_{L^2_t(L^2)}\bigr)\Bigr\}\\
\leq & Ct^{\f14}\d E(t).
\end{split}
\eeno Hence, we arrive at \ben\label{eq:g-time-3d} \|\p_t\d
g\|_{L^2_t(L^2)}\leq Ct^{\f14}\d E(t). \een Then
(\ref{eq:source-est}) follows from (\ref{eq:f1-3d}),
(\ref{eq:g-gradient}), (\ref{eq:f2-3d}) and
(\ref{eq:g-time-3d}).\end{proof}

\begin{proof}[Proof of (\ref{eq:source-est}) in 2-D case] In what
follows, we shall always  take $0<\al<s.$  By virtue of
(\ref{eq:A-expansion}) and (\ref{eq:regularity-v-2}), we have
 \beno
\|(Id-A_2^t)\na\d\Pi\|_{L^2_t(L^2)}\leq
C\int_0^t\|\na_yv_2(t')\|_{L^\infty}\,dt'\|\na\d\Pi\|_{L^2_t(L^2)}\leq
Ct^{\f{s}{2(1+\al)}}\d E(t), \eeno and by Gagliardo-Nirenberg
inequality, one has
\begin{align}
\|\d A\|_{L^\infty_t(L^{\f 2 \al})}\leq& C\bigl\|\int_0^t|\na\d
v|\,dt'\bigr\|_{L^\infty_t(L^{\f 2\al})}\nonumber\\
\le& C\int_0^t\|\na\d
v\|_{L^2}^{\al}\|\na^2\d v\|_{L^2}^{1-\al}\,dt'\le Ct^{\f{1+\al} 2}\d E(t),\label{eq:dA}
\end{align}
which along with \eqref{eq:regularity-v-2} implies that
\beno
\begin{split}
\|\d A^t\na \Pi_1\|_{L^2_t(L^2)}\leq &\|\tau^{-\f {1+\al-s} 2}\d
A(\tau)\|_{L^\infty_t(L^{\f 2\al})}\|\tau^{\f {1+\al-s} 2}\na
\Pi_1(\tau)\|_{L^2_t(L^\f 2 {1-\al})} \leq C t^{\f s 2}\d E(t).
\end{split}
\eeno As a consequence, we obtain \beq \label{eq:f1-2d} \|\d
f_1\|_{L^2_t(L^2)}\leq Ct^{\f{s}{2(1+\al)}}\d E(t). \eeq

While due to \eqref{eq:dA-gradient}, we get, by using (\ref{eq:dA})
and (\ref{eq:regularity-v-2}), that \beno
\begin{split}
\|\na(\d A\na v_1)\|_{L^2_t(L^2)} &\le
C\Bigl\{\bigl\|\int_0^\tau|\na^2\d v|\,d\tau'\na
v_1\bigr\|_{L^2_t(L^2)}+\bigl\|\int_0^\tau|\na\d v|\,d\tau'\na^2
v_1\bigr\|_{L^2_t(L^2)}\\
&\quad+\bigl\|\int_0^\tau|\na\d v|\,d\tau'\int_0^\tau\big(|\na^2
v_1|+|\na^2 v_2|\big)d\tau'\na
v_1\bigr\|_{L^2_t(L^2)}\Bigr\}\\
&\leq  C\Bigl\{\bigl\|\tau^{-\frac {1+\al-s}
{2(1+\al)}}\int_0^\tau|\na^2\d
v|\,d\tau'\bigr\|_{L^\infty_t(L^2)}\|\tau^{\frac {1+\al-s}
{2(1+\al)}}\na
v_1\|_{L^2_t(L^\infty)}\\
&\quad+\bigl\|\tau^{-\f{1+\al-s} 2}\int_0^\tau|\na\d
v|\,d\tau'\bigr\|_{L^\infty_t(L^{\f 2 \al})}\|\tau^{\f{1+\al-s}
2}(\na^2 v_1,\na^2v_2)\bigr\|_{L^2_t(L^{\f 2
{1-\al}})}\\
&\quad\times \big(1+\|\tau^{\frac {1+\al-s} {2(1+\al)}}\na
v_1\|_{L^2_t(L^\infty)}\big)\Bigr\}\\
&\le Ct^{\f{s}{2(1+\al)}}\d E(t),
\end{split}
\eeno
and we also have
\beno
\begin{split}
&\|\na\bigl((Id-A_2)\na\d v\bigr)\|_{L^2_t(L^2)}\\
&\leq
C\Bigl\{\bigl\|\int_0^t|\na^2
v_2|\,d\tau'\bigr\|_{L^\infty_t(L^{\f 2 {1-\al}})}\|\na\d
v\|_{L^2_t(L^{\f 2 \al})}
+\bigl\|\int_0^t|\na
v_2|\,d\tau'\bigr\|_{L^\infty_t(L^\infty)}\|\na^2\d
v\|_{L^2_t(L^2)}\Bigr\}\\
&\leq Ct^{\f{s}{2(1+\al)}}\d E(t).
\end{split}
\eeno Here we used the fact that \ben\label{eq:dv-Lp} \|\na\d
v\|_{L^2_t(L^{\f 2 \al})}\le Ct^{\f \al 2}\|\na\d
v\|_{L^\infty_t(L^2)}^{\al}\|\na^2\d v\|_{L^2_t(L^2)}^{1-\al} \le
Ct^{\f \al 2}\d E(t). \een Hence, we obtain \beq
\label{eq:g-gradient-2d} \|\na\dive\d g\|_{L^2_t(L^2)}\leq
Ct^{\f{s}{2(1+\al)}}\d E(t). \eeq

Now we handle $\|\d f_2\|_{L^2_t(L^2)}$. Indeed it follows from
(\ref{eq:regularity-v-2}) and (\ref{eq:dA}) that \beno
\begin{split}
&\|\dive\bigl[\bigl(\d A(Id-A_2^t)\bigr)\na v_1\bigr]\|_{L^2_t(L^2)}\\
&\leq \|\tau^{-\frac {1+\al-s} {2(1+\al)}}\na\d
A\|_{L^\infty_t(L^2)}\|Id-A_2\|_{L^\infty_t(L^\infty)}\|\tau^{\frac {1+\al-s} {2(1+\al)}}\na
v_1\|_{L^2(L^\infty)}\\
&\quad+\|\tau^{-\frac {1+\al-s} {2(1+\al)}}\d A\|_{L^\infty_t(L^{\f
2 \al})}\|\na A_2\|_{L^\infty_t(L^{\f 2 {1-\al}})}\|\tau^{\frac
{1+\al-s} {2(1+\al)}}\na
v_1\|_{L^2(L^\infty)}\\
&\quad+\|\tau^{-\f {1+\al-s} {2} }\d A\|_{L^\infty_t(L^{\f 2
\al})}\|Id-A_2\|_{L^\infty_t(L^\infty)}\|\tau^{\f {1+\al-s} {2}
}\na^2
v_1\|_{L^2(L^{\f 2 {1-\al}})}\\
&\le  Ct^{\f{s}{2(1+\al)}}\d E(t).
\end{split}
\eeno The same estimate holds for the remaining terms
in(\ref{eq:f2-1}). Hence, we get \beno
\|\dive\bigl[(A_2A_2^t-A_1A_1^t)\na v_1\bigr]\|_{L^2_t(L^2)}\leq
Ct^{\f{s}{2(1+\al)}}\d E(t). \eeno Whereas thanks to
(\ref{eq:f2-2}), we get, by using (\ref{eq:regularity-v-2}), that
\beno
\begin{split}
&\|\dive\bigl((Id-A_2A_2^t)\na\d v\bigr)\|_{L^2_t(L^2)}\\
&\leq C\bigl(\|\na A_2\|_{L^\infty_t(L^{\f 2
{1-\al}})}\|Id-A_2\|_{L^\infty_t(L^\infty)}\|\na\d v\|_{L^2_t(L^{\f
2 \al})}+\|(Id-A_2)\|_{L^\infty_t(L^\infty)}\|\na^2\d
v\|_{L^2_t(L^2)}\bigr)\\
&\leq Ct^{\f{s}{2(1+\al)}}\d E(t).
\end{split}
\eeno So we obtain \beq\label{eq:f2-2d} \|\d f_2\|_{L^2_t(L^2)}\leq
Ct^{\f{s}{2(1+\al)}}\d E(t). \eeq

Finally we deal with $\|\p_t\d g\|_{L^2_t(L^2)}$.  As in the 3-D
case, we have \beno
\begin{split}
\|\p_t\bigl[\d A v_1\bigr]\|_{L^2_t(L^2)}\leq & C\Bigl\{\|(\na\d v)
v_1\|_{L^2_t(L^2)}+\bigl\|\int_0^\tau|\na\d
v|\,d\tau'|(\na
v_1, \na v_2)|v_1\bigr\|_{L^2_t(L^2)}\\
&\quad+\bigl\|\int_0^\tau|\na\d
v|\,d\tau'|\p_tv_1|\bigr\|_{L^2_t(L^2)}\Bigr\}.
\end{split}
\eeno
By (\ref{eq:regularity-v-2}), the first term on the right hand side is bounded by
\beno
\|\na\d v\|_{L^\infty_t(L^2)}\|v_1\|_{L^2_t(L^\infty)}\le C\d E(t)\|v_1\|_{L^\infty_t(L^2)}^\f12\|\na^2v_1\|_{L^1_t(L^2)}^\f12\le Ct^{\f {s} 4}\d E(t),
\eeno
and the third term is bounded by
\beno
\bigl\|\tau^{-\f {1+\al-s} 2}\int_0^\tau|\na\d
v|\,d\tau'\bigr\|_{L^\infty_t(L^{\f 2\al})}\|\tau^{\f {1+\al-s} 2}\p_tv_1\|_{L^2_t(L^{\f 2 {1-\al}})}
\le Ct^{\f s 2}\d E(t),
\eeno
and the second term is bounded by
\beno
&&\bigl\|\tau^{-\f{1+\al-s} {2(1+\al)}-\f {(1-s)\al} 2}\int_0^\tau|\na\d
v|\,d\tau'\bigr\|_{L^\infty_t(L^{\f 2\al})}
\|\tau^{\f{1+\al-s} {2(1+\al)}}(\na
v_1, \na v_2)\|_{L^2_t(L^\infty)}\|\tau^{\f {(1-s)\al} 2}v_1\bigr\|_{L^\infty_t(L^\f 2 {1-\al})}\\
&&\le Ct^{\f {s(1+\al+\al^2)} {2(1+\al)}}\d E(t). \eeno On the other
hand, it follows from  Gagliardo-Nirenberg inequality that
\begin{align*}
\|\tau^{-\f {(1-s)\al} 2}\d v\|_{L^\infty_t(L^{\f 2 \al})}
\le& C\|\tau^{-\f {1-s} {2}}\d v\|_{L^\infty_t(L^2)}^{\al}\|\na\d v\|_{L^\infty_t(L^2)}^{1-\al}\\
\le& Ct^{\f {s\al} 2}\|\p_t\d v\|_{L^2_t(L^2)}^{\al}\|\na\d v\|_{L^\infty_t(L^2)}^{1-\al}
\le Ct^{\f {s\al} 2}\d E(t),\\
\|\tau^{\f {(1-s)\al} 2}\na v_2\|_{L^2_t(L^{\f 2 {1-\al}})}
\le& C\|\na v_2\|_{L^2_t(L^2)}^{1-\al}\|\tau^{\f {1-s} 2}\na^2 v_2\|_{L^2_t(L^2)}^\al\le C,
\end{align*}
from which and (\ref{eq:regularity-v-2}), we infer that
\beno
\begin{split}
&\|\p_t\bigl[(Id-A_2)\d v\bigr]\|_{L^2_t(L^2)}\\
&\leq C\Bigl\{\|\na
v_2\d v\|_{L^2_t(L^2)}+\bigl\|\int_0^\tau|\na v_2|\,d\tau'|\na
v_2||\d v|\bigr\|_{L^2_t(L^2)}
+\bigl\|\int_0^\tau|\na v_2|\,d\tau'|\p_t\d v|\bigr\|_{L^2_t(L^2)}\Bigr\}\\
&\leq C\Bigl\{\|\tau^{-\f {(1-s)\al} 2}\d v\|_{L^\infty_t(L^{\f 2 \al})}\|\tau^{\f {(1-s)\al} 2}\na v_2\|_{L^2_t(L^{\f 2 {1-\al}})}\\
&\qquad\quad+\|\na v_2\|_{L^1_t(L^\infty)}
\bigl(\|\tau^{-\f {(1-s)\al} 2}\d v\|_{L^\infty_t(L^{\f 2 \al})}\|\tau^{\f {(1-s)\al} 2}\na v_2\|_{L^2_t(L^{\f 2 {1-\al}})}+\|\p_t\d v\|_{L^2_t(L^2)}\bigr)\Bigr\}\\
&\leq  Ct^{\f {s\al} 2}\d E(t).
\end{split}
\eeno Therefore, we arrive at \beq \label{eq:g-time-2d} \|\p_t\d
g\|_{L^2_t(L^2)}\leq Ct^{\f {s\al} 2}\d E(t). \eeq Then
(\ref{eq:source-est}) in the 2-D case follows from
(\ref{eq:f1-2d})-(\ref{eq:g-time-2d}).
\end{proof}

\bigskip

\noindent {\bf Acknowledgments.} The authors would like to thank the
anonymous referees for many profitable suggestions. Part of this
work was done when M. Paicu was visiting Morningside Center of the
Chinese Academy of Sciences in the Spring of 2012. We appreciate the
hospitality of MCM and the financial support from the Chinese
Academy of Sciences. P. Zhang is partially supported by NSF of China
under Grant 10421101 and 10931007, the one hundred talents' plan
from Chinese Academy of Sciences under Grant GJHZ200829 and
innovation grant from National Center for Mathematics and
Interdisciplinary Sciences. Z. Zhang is partially supported by NSF
of China under Grant 10990013 and 11071007.

\end{document}